\documentclass[amsart,14]{article}

\usepackage{amsfonts}
\usepackage{amsmath}
\usepackage{amssymb}
\usepackage{amsthm}
\usepackage{gensymb}
\usepackage{geometry}
\usepackage{graphicx}
\usepackage[section]{placeins}
\usepackage{float}
\usepackage{cite}
\usepackage{setspace}
\usepackage{mathrsfs}
\usepackage{verbatim}
\usepackage{hyperref}
\usepackage{authblk}

\newcommand{\Z}{\mathbb Z}

\newcommand{\T}{^{\mathsf{T}}}
\newcommand{\C}{\mathbb C}
\newcommand{\R}{\mathbb R}

\newcommand{\bff}{\mathbf f}
\newcommand{\bfg}{\mathbf g}
\newcommand{\F}{\mathcal F}
\newcommand{\bF}{\mathbf F}
\newcommand{\bD}{\mathbf D}
\newcommand{\rmd}{\mathrm{d}}
\newcommand{\bvx}{\vec{\mathbf{x}}}
\newcommand{\vxt}{\vec{\xi_\theta}}

\newtheorem{theorem}{Theorem}

\newtheorem{lemma}{Lemma}
\newtheorem{proposition}{Proposition}
\newtheorem{definition}{Definition}

\title{Fourier Analysis, Computing, and Image Formation for Spotlight Synthetic Aperture Radar}

\author[1]{Toby Sanders}
\author[2]{Christian Dwyer}
\author[1]{Rodrigo B. Platte}

\affil[1]{School of Mathematical and Statistical Sciences, Arizona State University}
\affil[2]{Department of Physics, Arizona State University}

\date{}

\begin{document}
\maketitle
\begin{abstract}
This article is written to serve as an introduction and survey of imaging with synthetic aperture radar (SAR).  The reader will benefit from having some familiarity with harmonic analysis, electromagnetic radiation, and inverse problems.  After an overview of the SAR problem and some main concepts, the SAR imaging problem is contextualized in terms of classical harmonic analysis.  Within this context, we consider partial Fourier sums of off-centered Fourier data and correspondingly the convolutional kernels resulting from conventional SAR image formation techniques. Following this, we revisit imaging of random complex signals from frequency data as in SAR, providing simpler derivations of some previous results and extending these ideas to the continuous setting.  These concepts are tied in with the derived convolutional kernels, and it is deduced how good an image approximation is when it is obtained from only a small band of high frequency Fourier coefficients.  Finally, regularization methods are presented to improve the quality of SAR images.  Corresponding MATLAB software is made available for reproducibility of most figures and to facilitate further exploration of the methods presented here.
\end{abstract}


\section{Introduction and Basic Setup for SAR}

Synthetic aperture radar (SAR) is an all weather, day or night radar technique used for imaging targets in defense, security, geodesy, remote sensing, and other applications.  The imaging procedure requires solving an inverse problem that models the reflected radar data from the target scene \cite{Andersson, thompson1996spotlight}.

The data in spotlight SAR is collected by using an antenna to transmit microwaves into the imaging scene and measuring the reflected \emph{echo} response with a receiver.  This process is repeated by passing around the scene at a series of locations that may be specified by an azimuth angle $\theta$ and the elevation angle relative to the ground plane, $\phi$.  For simplicity  we assume a single antenna and receiver system attached to an aircraft passing around the scene.

\begin{figure}
 \centering
 \includegraphics[width=.56\textwidth]{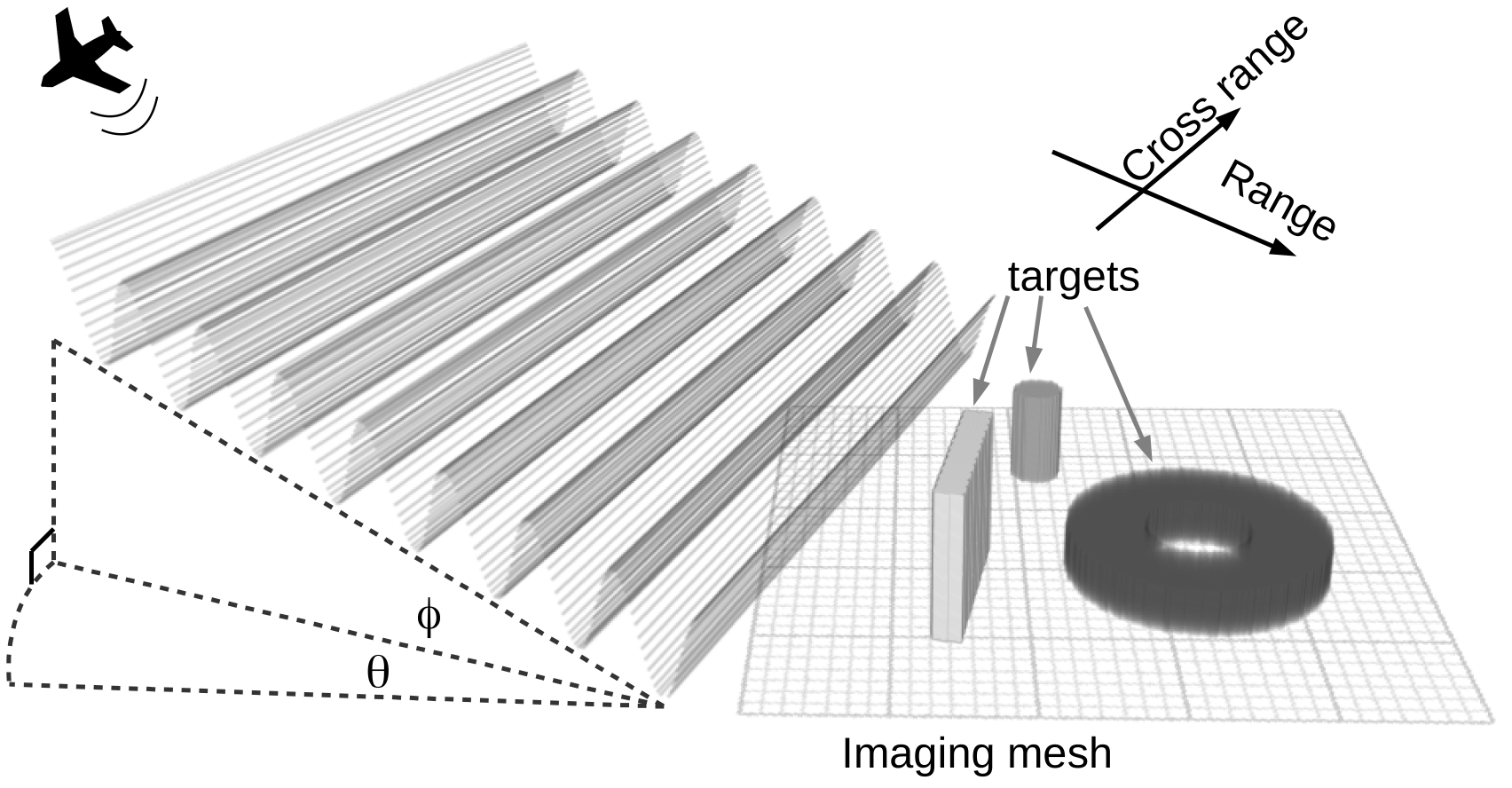}
 \includegraphics[width=.4\textwidth]{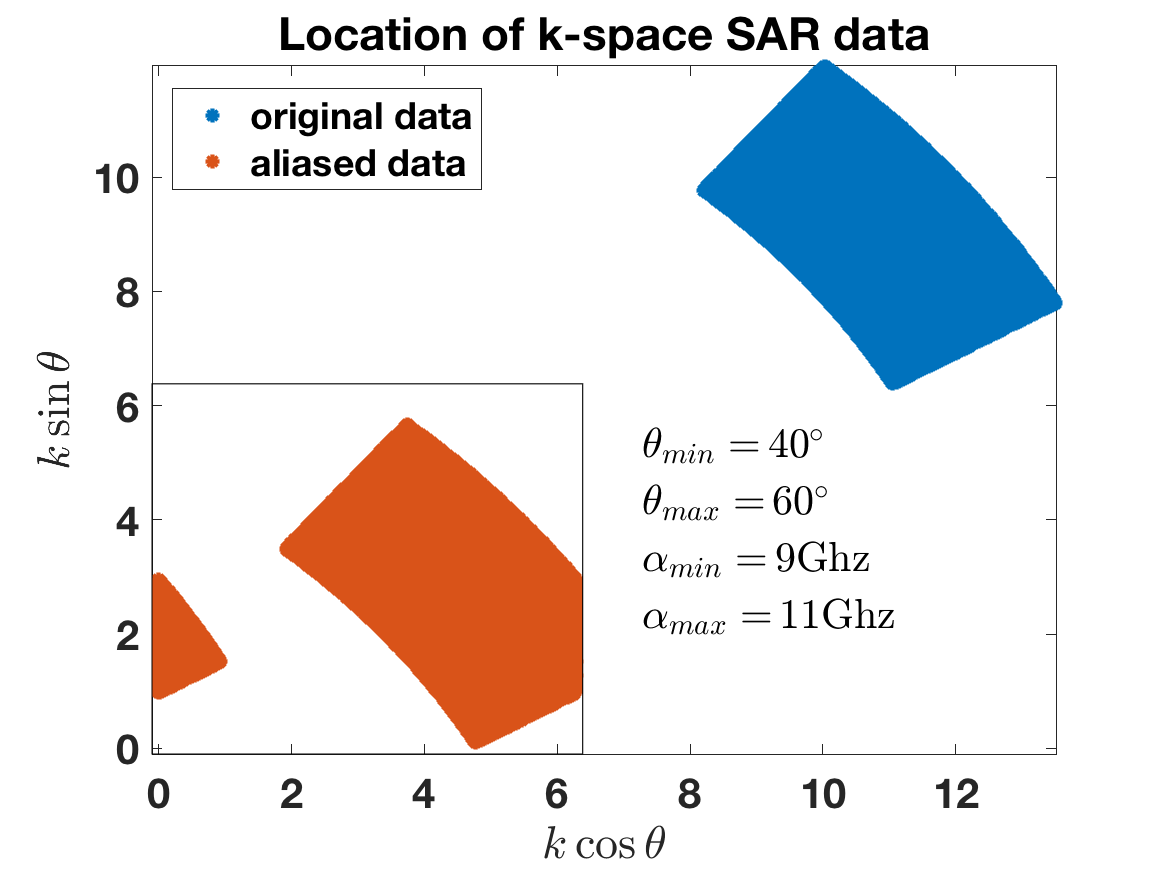}
 \caption{Left: visual of the SAR data acquisition.  Right: location of the SAR data in Fourier space, and the aliased data (orange).}
 \label{fig: 3D}
\end{figure}

To describe the data, we introduce the most commonly used Fourier model for practical SAR imaging.  It is important to note that this approximation model for SAR has received a great deal of attention in the mathematical literature \cite{cheney2001mathematical,nolan2002synthetic,cheney2009fundamentals,gilman2015mathematical}, which may be observed as a result of the first order Born approximation of the general model.  As an alternative to these endeavors, the work here focuses on the numerical imaging and approximation aspects as a result of the given model.  To this end, let the transmitted microwave frequencies be given by $\{\alpha_j\}_{j=1}^M$, e.g. equally spaced frequencies with a center near the 10 GHz frequency.  Let the imaging scene be a field of scattering objects or reflectors, which we denote by $f(x,y)$. Then the data in SAR, from azimuth angle $\theta$ and an angle of elevation $\phi$, is modeled as
\begin{equation}\label{eq1}
\begin{split}
\hat f(k,\theta) 
& = \hat f(k(\alpha_j, \phi),\theta) = \iint_{\Omega} f(\bvx) \exp\left( -ik \, \vxt \cdot \bvx \right) \, \rmd A\\
& = \iint_{\Omega} f(x,y) \exp\left(-ik(\alpha_j,\phi) (\cos \theta, \sin \theta)\cdot (x,y)\right) \, \rmd A 
\end{split}
\end{equation}

The notation used here and continued through the document is $\vxt = (\cos\theta , \sin\theta)$ and $\bvx = (x,y)$.  The Fourier frequencies $k = k(\alpha_j , \phi)$ are determined by the microwave frequencies $\alpha_j$ (typically measured in GHz) and the angle of elevation $\phi$.  Then the general image reconstruction problem in SAR is to approximate $f(x,y)$ from the data set $\{ \hat f(k_j,\theta_i)\}_{j,i=1}^{M,P}$.  The frequencies are typically equally spaced as well as the azimuth angles.  For simplicity, we will assume a constant angle of elevation.  This set up is depicted on the left of Figure \ref{fig: 3D}.  The right of Figure \ref{fig: 3D} shows the nodes on a polar grid in Fourier space on which the data is located, and it also shows the aliasing of the nodes (orange).  This aliasing is described in section \ref{sec: alias}.

We proceed by reformulating the data within the tomographic framework, which will be useful for some of the work here.  To do so, simply make the substitution $(w,z) = (x,y) Q_\theta$, where $Q_\theta$ denotes the standard rotation matrix by angle $\theta$.  Then if we let $\Omega$ be the ball of radius $R$ centered at the origin, (\ref{eq1}) becomes
\begin{equation}\label{tomo}
\hat f (k , \theta ) =  \int_{-R}^R \int_{-\sqrt{R^2 - w^2}}^{\sqrt{R^2 - w^2 }} f((w,z)Q_\theta^T) \, dz \exp \big\{-ik w \big\} \, dw .
\end{equation}
Those familiar with tomography will notice the inner integral in (\ref{tomo}) over $z$ defines $p_\theta f$, the projection or Radon transform of $f$ at angle $\theta$.  Therefore, (\ref{tomo}) essentially represents the well-known Fourier slice theorem, which is sometimes used to depict SAR using a tomographic formulation \cite{munson1983tomographic,Jakowatz1996,Nat-Cheney}.  We point this out since some of our discussions involving the data processing for SAR will be simplified to 1D Fourier integrals represented by the outer integral over $w$ in (\ref{tomo}), i.e.  
\begin{equation}
\hat f(k,\theta) = \int_{-R}^R p_\theta f ( w) e^{-ik w } \, dw .
\end{equation}
In SAR, these 1D profiles or projections $p_\theta f$ are often referred to as range profiles.

Figure \ref{fig: example} highlights some of these ideas from the openly available CV dome data \cite{dungan2010civilian}, where some of the concepts presented are explained in more detail in the proceeding sections.  The full 360$\degree$ azimuth data was utilized to generate the image, as well as the full available bandwidth of data as shown in the far left image.  The image formation was performed with an approximation of the matched filter using a nonuniform FFT \cite{Greengard, Andersson, 1166689}.  The middle image shows the 1D inverse Fourier transform of the columns of the data shown of the left, which forms the so-called range profiles of $f$, and may be considered high pass filtered projections.

\begin{figure}
 \centering
 \includegraphics[width=1\textwidth]{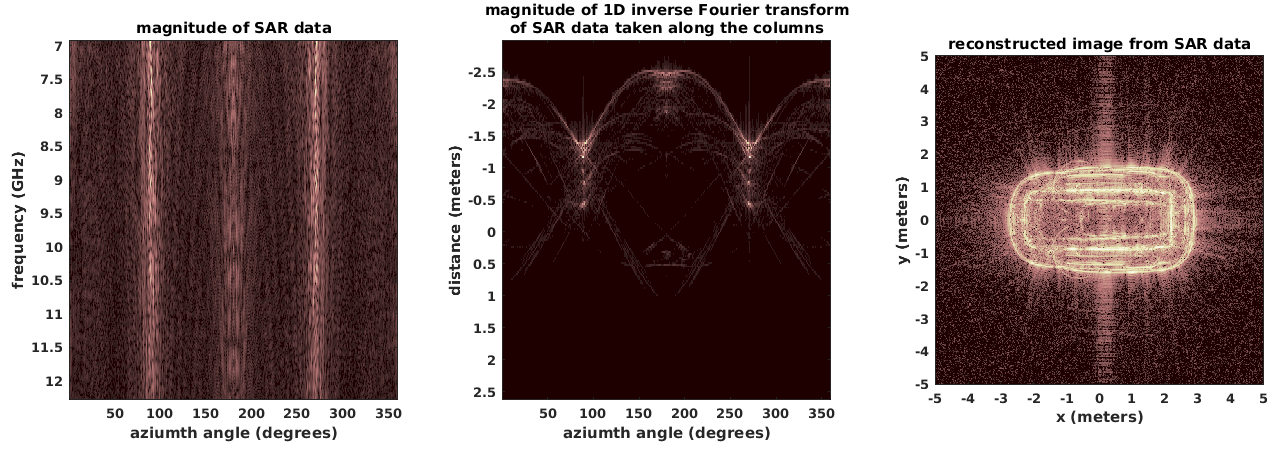}
 \caption{Visualization of synthetic SAR data of a jeep and corresponding reconstruction.}
 \label{fig: example}
\end{figure}

With these main ideas in order, the remainder of this document serves as a survey of many important practical and theoretical aspects of SAR imaging.  In section \ref{sec:freqs}, some fundamental concepts of SAR imaging are provided for background and completeness, including the connection between continuous and discretized domains, aliasing of frequencies, and standard image formation techniques.  In section \ref{sec:filter}, a more theoretical perspective on these standard imaging techniques is analyzed, particularly through an extension of some concepts in classical Fourier analysis and corresponding convolution kernels arising in SAR.  In section \ref{sec:rand} the nature of random phases in SAR images is revisited and extended, and it is shown how this effects both positively and negatively the results with the derived convolutional kernels derived in the previous section.  Finally, we complete this work by outlining some regularization techniques that can improve upon the standard imaging techniques, which are motivated by the results in all previous sections.  Throughout each section many examples and figures are provided to aid in the discussion.  Finally, most of the software for the techniques described here is made available, including reproducible MATLAB code for most figures. Downloading additional open access MATLAB software is necessary for successful implementation these codes in many cases, and can be found in \cite{toby-web,1166689}.

\section{Fundamentals of SAR Imaging}\label{sec:freqs}

Formally, unprocessed SAR measurements are reflections that result from mixing or convolving the scene with linear chirp waveforms, 
and therefore require simple processing steps for making the reflected signal approximate Fourier coefficients in the form (\ref{eq1}) (see \cite{thompson1996spotlight}, pp. 1--31).  From the raw reflections resulting from convolution of the scene with the linear chirp waveform, information from the instantaneous frequencies are ``pulled apart'' or \emph{demodulated} so long as the linear chirp rate is small enough.  The resulting Fourier frequencies or $k$ values are given by
\begin{equation}\label{k-values}
k = k(\alpha_j , \phi) = 2\pi \cos \phi \frac{2\alpha_j}{c} ,
\end{equation}
where $c$ is the speed of light in a vacuum approximately given by $\sim 3\times 10^8\text{m/s}$.
The factor of $\cos \phi$ comes from an effective \emph{shortening} of the frequency relative to the ground plane at higher elevations.  The remaining factors are standard for converting electromagnetic frequencies to their respective wavelengths or physical frequencies, aside from the factor of two, which is a result of the back scattering time delay to the receiver.  The typical microwaves frequencies $\alpha_j$ in SAR data acquisition are near the 10~GHz band, and therefore by (\ref{k-values}) the measured wavelengths are around $\sim 1.5$~cm, which gives us some idea of the maximum possible resolution (minimum pixel size) in the reconstruction.

\subsection{Scene Size}\label{sec: size}
\begin{figure}
 \centering
 \includegraphics[width=.8\textwidth]{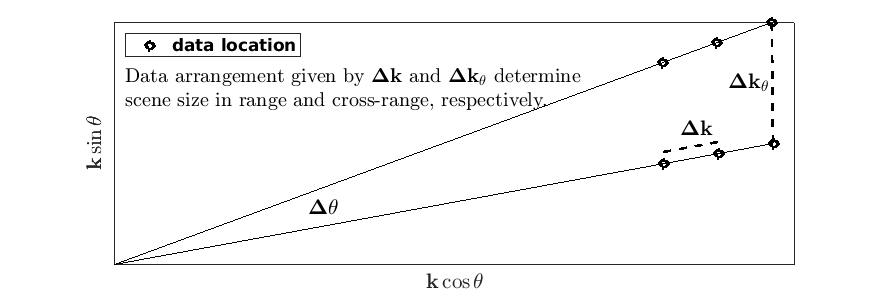}
 \caption{Demonstration of maximum scene size determined by data arrangement.}
 \label{fig: scene-size}
\end{figure}

Here we consider the appropriate scene size for reconstruction based on the SAR data parameters.  For further details on these topics see \cite{cheney2001mathematical,Nat-Cheney,SARMATLAB,cutrona1990synthetic}.  The scene size for the Fourier series depends on $\Delta k $, the difference between the frequencies.  In particular, for a scene of length $2R$ we have $\Delta k = \pi / R$, i.e. $R = \pi / \Delta k$.  Letting $\Delta \alpha$ denote the difference between the microwave frequency values, then by (\ref{k-values}) the maximum alias free scene radius is
\begin{equation}\label{scene-size}
 R = \frac{\pi}{\Delta k}  = \frac{c}{4\Delta \alpha \cos \phi}.
\end{equation}
However, this is only for the range direction parallel to the incident wave (see Figure \ref{fig: 3D}).  For SAR, there is a spacing of the frequencies in the range direction given by $\Delta k$, and spacing in frequencies in the cross-range direction, which is determined by the angle increment $\Delta \theta$.  See Figure \ref{fig: scene-size} for a visualization, where the apparent size of $\Delta \theta$ is exaggerated for visual clarity.  Assuming the spacing to be sufficiently small and using the approximation $\sin \left(\Delta \theta\right)  \approx \Delta \theta$, one obtains the increment in the frequencies in the cross-range or angular direction to be 
$$\Delta k_\theta = k_M \frac{\sin \Delta \theta}{\sin((\pi-\Delta \theta)/2)} \approx k_M  \Delta \theta,
$$
where $k_M$ is the largest frequency given by (\ref{k-values}).  Hence one defines the maximum alias free cross-range radius by
$$
R_{cr} = \frac{\pi}{k_M \Delta \theta} = \frac{c}{4 \alpha_M \cos \phi \Delta \theta}.
$$
Typically the parameters $\Delta \theta$ and $\Delta k$ are chosen to match the scene to be imaged \cite{SARMATLAB}, barring any limiting constraints arising from the hardware or physics.  For further details on the limits of the scene size based on the data geometry, see for example \cite{rigling2005taylor}.

Denote the pixel length by $h = 2R/N$, where $N$ is the number of pixels.  For convenience we can define the dimensionless spatial frequencies
\begin{equation}\label{fund-freq}
 (k_1 , k_2) = \vec{k}(\alpha_j , \phi , \theta, h) =  2\pi h\cos\phi\frac{2\alpha_j}{c} (\cos \theta , \sin \theta).
\end{equation}
Then for $f\in \C^{N\times N}$ on a uniform symmetric mesh, the 2-D discrete Fourier coefficients of frequency $(k_1,k_2)$ are defined by
\begin{equation}\label{2D-DFT}
 \hat f_{k_1 , k_2} = \sum_{j_1, j_2=-N/2}^{N/2-1} f_{j_1 , j_2} e^{-i(k_1,k_2)\cdot (j_1 , j_2)} .
\end{equation}
Hence, given the GHz frequencies $\alpha_j$, the angle(s) of elevation $\phi$, the sampling azimuth angles $\theta$, and a pixel length $h$, the data when interpreted as Fourier coefficients of the DFT in (\ref{2D-DFT}) has frequency values as written in (\ref{fund-freq}).

\subsection{Aliasing of the Frequencies}\label{sec: alias} 
Consider the 1-D Fourier representation of a function $f\in L_2[-R,R]$ (for additional details, see the appendix), as well as its discretized DFT.  The Fourier representation of $f$ on an uniform $N$ point mesh, i.e. $f_j = f\left(x_j \right)$, where $x_j =  2jR/{N} = jh$, for $~ j=-N/2,-N/2+1,\dots, N/2-1$.  Then
\begin{equation}\label{alias-1}
\begin{split}
f_j = \sum_{k\in \Z} \hat f_k \exp\left( {i\tfrac{2 \pi }{2 R} k x_j }\right)
 = \sum_{k\in \Z} \hat f_k \exp\left( {i\tfrac{2\pi }{N} k j }\right)
\end{split} 
\end{equation}
Proceed by rewriting this sum as a double sum by separating it into groups of $N$:
\begin{equation}\label{alias-2}
\begin{split}
f_j 
&= \sum_{m\in \Z } \sum_{k=0}^{N-1} \hat f_{k+mN} \exp\left({i\tfrac{2\pi}{N} (k+mN) j} )\right)\\
& = \sum_{k=0}^{N-1} \sum_{m\in\Z} \hat f_{k+mN} \exp\left(i\tfrac{2\pi}{N} kj\right) \underbrace{\exp\left( i\tfrac{2\pi}{N}mNj )\right) }_1 \\
& = \sum_{k=0}^{N-1} \exp\left(i\tfrac{2\pi}{N} kj\right)  \underbrace{\left[\sum_{m\in\Z} \hat f_{k+mN} \right]}_{\hat F_k} \\
& = \sum_{k=0}^{N-1} \hat F_k \exp\left(i\tfrac{2\pi}{N} kj\right) 
\end{split}
\end{equation}
The interpretation of the last line is that on this uniform grid one only needs $N$ waves for the representation (as one should expect, since the waves form an orthonormal basis for $\C^N$).  Moreover, the coefficients in this representation, $\hat F_k$, are given by an infinite sum of equally spaced coefficients of the continuous $f$.  This is a typical example of aliasing.  

This additionally implies that any frequency, $2\pi k/N$ in (\ref{alias-1}), which exceeds $2\pi$ will be modulated by $2\pi$.  One could mitigate this effect by decreasing $h$ in (\ref{fund-freq}), i.e. increasing $N$, which means we would be able to \emph{see} higher frequencies with more pixels.  In our experience this does not lead to any notable improvement in image quality, assuming $N$ is already sufficiently large.  


\subsection{Standard Image Formation Procedures}\label{sec: imform}
Most standard image formation procedures for SAR involve some numerical technique that approximates an inverse Fourier transform.  Let's assume the data takes the form in (\ref{eq1}), with the values $k_j  = 4\pi h \cos \phi \frac{\alpha_j}{c}$ as in (\ref{fund-freq}) for $\{ \alpha_j \}_{j=1}^M$ and for a set of angles $\Theta = \{ \theta_1, \theta_2, \dots, \theta_P\}$.  Then the so called matched filter reconstruction is given by 
\begin{equation}\label{mf}
 f_{MF}(x,y) =  \sum_{\theta \in \Theta} \sum_{j=1}^M \hat f(k_j, \theta) \exp\left( ik_j \vxt \cdot \bvx \right),
\end{equation}
where in practice $\bvx = (x,y)$ would be defined over the discrete mesh.  Computing this directly is computationally burdensome ($O(MPN^2)$, where $N^2$ is the number of pixels), and therefore most conventional techniques try to approximate (\ref{mf}) in some efficient way.  

The most popular approach attempts to make use of the efficient FFT.  The FFT is in its basic case is limited to Fourier coefficients defined over an equally spaced Cartesian grid, yet the Fourier coefficients in SAR are spaced over a polar grid.  Therefore the general strategy is to \emph{regrid} the data onto an equally spaced Cartesian grid with a well-designed interpolation scheme and then apply an FFT.  These methods include the polar format algorithm and the nonuniform FFT (NUFFT), which may go by other names within the literature.  There is a vast amount of literature devoted to this subject, the details of which go beyond the scope of this article.  Andersson et. al. summarize and compare these methods in \cite{Andersson}.  

Another alternative is the backprojection algorithm, which is also an accelerated algorithm to approximate (\ref{mf}).  This approach makes use of the Fourier slice theorem in (\ref{tomo}), and is typically slower than a NUFFT algorithm.  The appeal however is in the possibility of implementing backprojection on the fly as new data is acquired and simply adding the data to the reconstruction.  Additionally, it is easily parallelizable making it attractive for GPU computations.  Finally, backprojection is also a useful educational tool to begin working in SAR for those unfamiliar. To perform backprojection, first 1-D Fourier transforms are applied to each angle $\theta$ or pulse of data $\{\hat f (k,\theta ) \}_k$ to yield approximate high pass filtered projections of $f$ denoted by $p_\theta f$ (see Theorem \ref{cor1} for a formal definition, and the middle column of Figure \ref{fig: example} for a visual).  This projection $p_\theta f$ is usually referred to as the range profile of $f$ in SAR.  The range profiles are then \emph{backprojected} onto the image mesh using a fast interpolation scheme.  An implementation and description of backprojection are given by Gorham and Moore \cite{SARMATLAB}, along with supporting MATLAB software.

Finally, due to noise and imperfections in the model, finding solutions to approximate (\ref{mf}) generally results in subpar image quality.  The most common approach is then to find a particular \emph{smooth} solution that satisfies the data to a lesser extent.  This can be done using regularization techniques, and if done properly, can be computed nearly as efficiently as say backprojection \cite{sanders2017composite}. General models for regularization are given later in section \ref{sec: reg}.


\section{Convolution Kernels Arising from Conventional SAR Image Formation}\label{sec:filter}
In this section we characterize the function approximation of $f$ resulting from (\ref{mf}).  In particular, we arrive at a kernel $\mathcal K$, such that $f_{MF}(x,y) = f*\mathcal K(x,y)$, where $\mathcal K$ depends on the sampling parameters.  While it may appear that this kernel may only be useful for capturing edges of $f$ from typical SAR sampling parameters, section \ref{sec:rand} will show that this is not the case, and indeed much more information is acquired in the high band frequency data.  The derivation of the kernel $\mathcal K$ is of the flavor of classical Fourier analysis results, and so we begin by first introducing some of these concepts initially for those that may be unfamiliar.  For reference on the proceeding discussion around (\ref{SN})-(\ref{SNsigma}), see for example chapter 15 in \cite{bruckner1997real}.

An important concept in Fourier analysis and imaging is the approximation of a finite or truncated sum to an infinite Fourier series, e.g. the finite set of data one is capable of acquiring in application.  A classical result is the determination of the partial Fourier sum
\begin{equation}\label{SN}
S_n f(x) = \sum_{k=-n}^n \hat f_k e^{ikx},
\end{equation}
where here we are using the convention that $f$ is $2\pi$-periodic and 
$
\hat f_k = \tfrac{1}{2\pi} \int_{-\pi}^\pi f(x) e^{-ikx} \, \rmd x .
$
It is well-known that this partial Fourier sum results in
\begin{equation}\label{dirichlet}
 S_n f(x) = f*D_n(x), \quad \text{where} \quad D_n(x) = \sum_{k=-n}^n e^{ikx} = \frac{\sin \left( (n+1/2) x \right)}{ \sin\left(x/2 \right)} ,
\end{equation}
and $D_n$ is known as the Dirichlet kernel.  By Parseval's theorem, evaluating this sum minimizes the $L_2$ error in the sense that if we let $S_n^\sigma f(x) = \sum_{k=-n}^n  \sigma_k \hat f_k e^{i kx}$, then
\begin{equation}\label{SNsigma}
 \arg \min_\sigma ||S_n^\sigma f(x) - f(x)||_2
\end{equation}
is achieved when $\sigma$ is the vector of all ones.  On the other hand, it does not necessarily minimize the pointwise error and results in undesirable Gibbs or ringing artifacts near jumps.  To alleviate these artifacts, the $\sigma$ weights may be dampened near the ends of the sum.  For instance, using the linear weights $\sigma_k = (1-\frac{|k|}{n})$ results in $S_n^\sigma f(x) = f*F_n(x)$, where $F_n$ is a more localized kernel known as the Fejer kernel.  The better localization of the kernel results in reduced pointwise error and pointwise convergence.  There exists a number of such filters for these weights, e.g. Hann, Hamming, and Gaussian windows, that all produce moderate improvements.  In section \ref{sec: reg} we provide more discussion on this and regularization techniques to reduce ringing artifacts.

\subsection{1D Kernels}
In SAR the gathered set of finite frequency data are not centered at the origin, and the sort of partial sums arising take the form
\begin{equation}\label{dirichlet-off}
 S_{K_1,K_2} f(x) = \sum_{k=K_1}^{K_2} \hat f_k e^{i k x},
\end{equation}
for some integers $K_1<K_2$.  In the context of SAR this sum may be interpreted as the partial sum or inverse Fourier transform of the line data at a particular angle $\theta$ to form the range profile or $p_\theta f$.  Hence the 1D range profiles from which the image is formed may contain the Gibbs ringing, which will result in ringing artifacts in the final 2D image (see Figure \ref{fig: example}).  In the proceeding set of results, we develop these concepts and generalize these ideas to the 2D sampling modality.

The following proposition states that this offset Fourier sum results in a convolution with a Dirichlet kernel multiplied by a phase term, which varies at a rate proportional to the distance of the central frequency from the origin.  
\begin{proposition}\label{prop:1}
Let $f\in L^2(-\pi ,\pi)$, $K_1,K_2\in \Z$, and consider the partial sum in (\ref{dirichlet-off}).  Define the central frequency as $K_c = (K_1+K_2)/2$ and the bandwidth $B = K_2-K_1$,  then the partial sum in (\ref{dirichlet-off}) results in 
\begin{equation}\label{dirichlet-off-2}
S_{K_1,K_2} f(x) = f* G(x;K_c, B) , ~ \text{where} ~ G(x; K_c, B) = e^{iK_cx} D_{B/2} (x).
\end{equation}
\end{proposition}
\begin{proof}
First rewrite (\ref{dirichlet-off-2}) in its integral form
\begin{equation}
  S_{K_1,K_2} f(x) = \sum_{k=K_1}^{K_2} \frac{1}{2\pi} \int_{-\pi}^\pi f(y) e^{-i ky} \, dy e^{i k x} = \frac{1}{2\pi}\int_{-\pi}^\pi f(y) \left[\sum_{k=K_1}^{K_2} e^{i k (x-y) }\right] \, dy,
\end{equation}
Hence the result is the convolution of $f$ with $\sum_{k=K_1}^{K_2} e^{ik x}$, and it suffices to show that $\sum_{k=K_1}^{K_2} e^{i k x} =  e^{iK_c t}D_{B/2}(x)$.  The calculation can be performed as follows:
\begin{align*}
 \sum_{k=K_1}^{K_2} e^{i k x} 
 & = e^{iK_c x} \sum_{k=K_1}^{K_2} e^{i(k-K_c)x}\\
 & = e^{iK_c x} \sum_{k=-B/2}^{B/2} e^{ikx}.
\end{align*}
Observing the last sum to be $D_{B/2}(x)$ completes the proof.
\end{proof}
\begin{figure}
 \centering
 \includegraphics[width=1\textwidth]{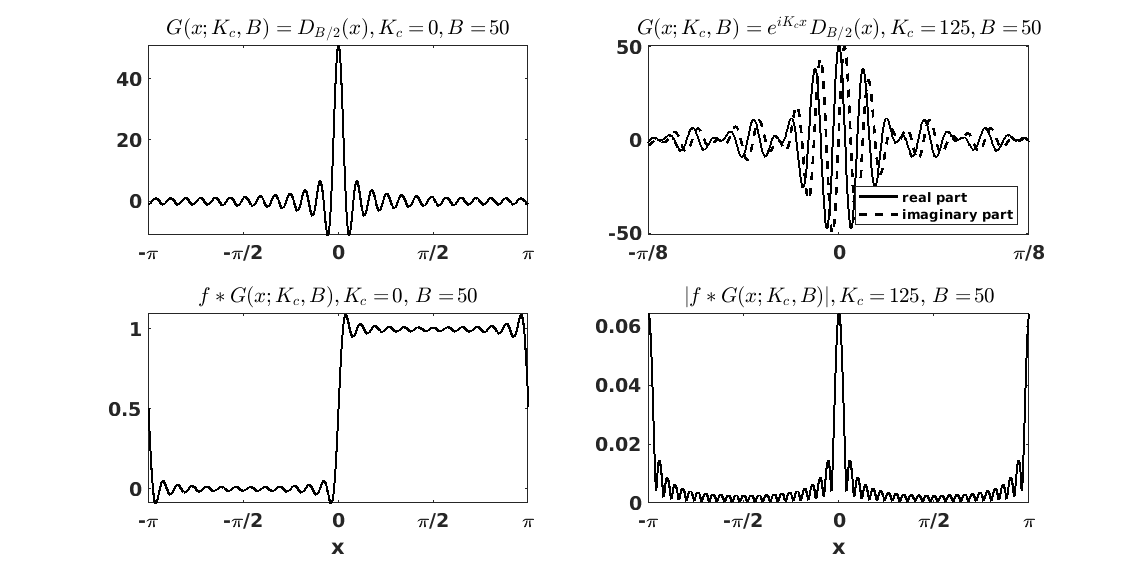}
\caption{Left: Real valued Dirichlet kernel with $K_c = 0$, and the convolution the kernel with a simple step function (bottom).  Right: Complex valued Dirichlet kernel with $K_c = 125$, and the convolution of the kernel with the same step function (bottom).}
\label{fig: 1Dconv}
 \end{figure}

This complex valued version of the Dirichlet kernel is shown in the top right of Figure \ref{fig: 1Dconv} along with its real counterpart in the top left, and the convolution with the function 
\begin{equation*}
 f(x) =
 \begin{cases}
      0 ,& x \in [-\pi , 0 ) \\
      1, & x\in [0,\pi )
  \end{cases}.
\end{equation*}
Observe the Gibbs phenomenon in both convolutions, and the complex version only detects the edges.  To formally connect this with SAR, we first provide the following lemma.  

\begin{lemma}\label{cor1}
Suppose for some angle $\theta$ we have data of the form (\ref{eq1}), for $k = k_1, \dots , k_M$, , with $k_j = k_1 + (j-1)\Delta k$ for each $j$, and $\Omega = \{(x,y) \, |\, x^2 +y^2 \le R^2\}$.  If we define the central frequency $K_c = (k_1+k_M)/2$, then evaluating the partial Fourier sum as in (\ref{dirichlet-off}) yields
\begin{equation}\label{dirichlet-off-3}
 \sum_{k = k_1}^{k_M} \hat f(k , \theta) e^{ikx} = S_{k_1,k_M} {p_\theta f } (x)
= p_\theta f * H (x; K_c , M, \Delta k ),
\end{equation}
where 
\begin{equation}\label{H-kernel}
H(x; K_c, M, \Delta k)= e^{iK_c x} D_{\frac{M-1}{2}}( \Delta k x ),
\end{equation}
and
\begin{equation}\label{proj-def}
 p_\theta f(x) = \int_{-\sqrt{R^2 - x^2}}^{\sqrt{R^2-x^2}} f((x,y)Q_\theta\T ) \, \rmd y , \quad  Q_\theta = 
 \left[
\begin{array}{cc}
\cos \theta & -\sin \theta\\
\sin \theta & \cos \theta
\end{array}
\right] .
\end{equation}
\end{lemma}
\begin{proof}
Writing the sum in (\ref{dirichlet-off-3}) in integral form leads to 
\begin{equation}
\sum_{j=1}^M \iint_{\Omega} f(w , z) e^{-ik_j (w,z)\cdot (\sin \theta , \cos \theta) }\, \rmd w \, \rmd z \, e^{ik_j x} =
\iint_{\Omega} f(w, z) \sum_{j=1}^M e^{-ik_j ( w,z)\cdot (\sin \theta , \cos \theta)} e^{ik_j x}  \, \rmd w\, \rmd z 
\end{equation}
Making the change of variables $ (w  , z) Q_\theta = (\alpha , \beta ) $, where $Q_\theta$ is the standard rotation matrix, yields
\begin{equation}
\iint_\Omega f((\alpha , \beta)Q_\theta^T) \left[ \sum_{j=1}^M e^{ik_j (x-\alpha)} \right] \, \rmd\alpha \, \rmd\beta = \iint_\Omega f((\alpha , \beta)Q_\theta^T) \rmd\beta \left[ \sum_{j=1}^M e^{ik_j (x-\alpha)} \right] \, \rmd\alpha 
\end{equation}
Evaluating the integral over $\beta$ reduces this expression to
\begin{equation}
\int p_\theta f (\alpha) \left[ \sum_{j=1}^M e^{ik_j (x-\alpha)} \right] \, \rmd\alpha ,
\end{equation}
hence for the remainder of the proof it suffices to show 
\begin{equation}\label{kernel-new}
\sum_{j=1}^M e^{ik_j x} = H (x; K_c, M , \Delta x) = e^{iK_c x} D_{\frac{M-1}{2}}( \Delta k x ).
\end{equation}
This can be done in a similar fashion used to prove Proposition \ref{prop:1}.
\end{proof}

\subsection{2D Kernels}
In this section, the 1D kernels derived in the previous section are extended to 2D sampling in SAR. Before arriving at the main result, the following lemma is needed.

\begin{lemma}\label{proj-lem}
Consider the operator $p_\theta$ acting on $f(x,y) \in L^2(\Omega)$ as defined in (\ref{proj-def}).  Then for any 1D function $g(x) : [-R,R]\rightarrow \C $, the adjoint operator of $p_\theta$ is given by 
\begin{equation}
p_\theta^* g(x,y) = g(\vxt \cdot (x,y)).
\end{equation}
\end{lemma}
\begin{proof}
For all $f(x,y)$ and $g(x)$ the adjoint operator satisfies
$$
\langle p_\theta f , g \rangle = \langle f ,p_\theta^* g \rangle.
$$
Starting from the left side we obtain
\begin{equation}
\begin{split}
\langle p_\theta f , g \rangle 
& = \int_{-R}^R p_\theta f(x) \overline{g(x)} \, \rmd x 
 = \iint_\Omega f((x,y) Q_\theta\T ) \, \rmd y \, \overline{g(x)} \, \rmd x \\
& = \iint_\Omega f(x',y') \overline{g((\cos \theta, \sin \theta)\cdot (x' , y') )} \, \rmd x' \rmd y'\\
& = \iint_\Omega f(x',y') \overline{p_\theta^*g(x',y') }\, \rmd x' \rmd y' =  \langle f ,p_\theta^* g\rangle
\end{split}
\end{equation}
\end{proof}

In words, $p_\theta$ is the operator which projects a 2D function into a 1D function by integrating in the direction of the angle $\theta$.  Likewise, the corresponding adjoint operator $p_\theta^*$ takes a 1D function and \emph{unwraps} the function to a 2D function by stretching it in the direction of angle $\theta$. This process can be called the backprojection of a single projection.  Figure \ref{fig: proj-example} provides a visual for these operators.

\begin{figure}
\centering
\includegraphics[width=1\textwidth,trim={7cm 0cm 7cm 0cm},clip]{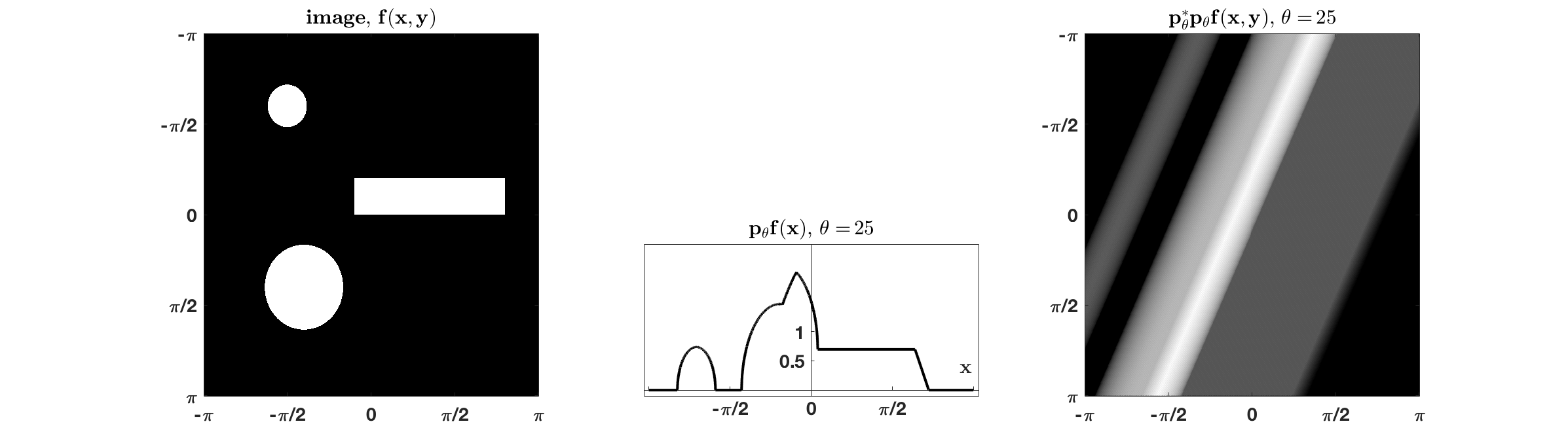}
\caption{An example of the projection of an image (left), $p_\theta f$ (middle), and applying the adjoint or \emph{backprojection} by $p_\theta^* p_\theta f$ (right).}
\label{fig: proj-example}
\end{figure}
\begin{theorem}\label{main-theorem}
Consider the conditions of Lemma \ref{cor1}.  Then the matched filter reconstruction given by (\ref{mf}) takes the form
\begin{equation}\label{mfconv}
 f_{MF} (x,y) = \sum_{\theta \in \Theta} p_{\theta} f * H(\vxt \cdot \bvx ; K_c, M , \Delta k ),
\end{equation}
where $H$ and $p_\theta f$ are given in Lemma \ref{cor1}.  Moreover, the matched filter reconstruction can be written as the convolution of $f$ with a point spread function or kernel by
$$
f_{MF}(x,y) = f* \mathcal K(x,y; K_c , M , \Delta k , \Theta ) ,
$$
where
\begin{equation}\label{main-2D-conv}
\mathcal K(x,y; K_c, M, \Delta k , \Theta ) = 
\sum_{\theta \in \Theta}  H ( \vxt \cdot \bvx ; K_c , M , \Delta k).
\end{equation}
\end{theorem}
\begin{proof}
By Lemma \ref{proj-lem} and equation (\ref{mf}), notice the matched filter response can be written as
 $$
 f_{MF}(x,y) = \sum_{\theta\in \Theta} p_{\theta}^* \left[ \sum_{j=1}^{M}  \hat f(k_j,\theta) \exp(ik_j x) \right].
 $$
 The resulting convolution from the inner sum is given by Lemma \ref{cor1}, and this completes the proof of (\ref{mfconv}).

 To prove (\ref{main-2D-conv}), we write (\ref{mfconv}) in its integral form:
  \begin{equation}
 \begin{split}
 f_{MF}(x,y) & = 
 \sum_{\theta \in \Theta} \int_{-R}^R p_\theta f(w) H(\vxt \cdot \bvx - w) \, \rmd w\\
 & = \sum_{\theta \in \Theta} \iint_\Omega f((w,z) Q_\theta\T ) \, \rmd z  \, H(\vxt \cdot \bvx  - w) \, \rmd w\\
 & =  \sum_{\theta \in \Theta} \iint_\Omega f(w',z' )   H(\vxt \cdot \bvx  - (w',z') ) \, \rmd z' \rmd w' \\
   & =  \iint_\Omega f(w',z' )   \left[ \sum_{\theta \in \Theta} H(\vxt \cdot \bvx  - (w',z') )\right] \, \rmd z' \rmd w' ,
  \end{split}
  \end{equation}
  where in the second line we used the definition of $p_\theta f$, and in the third line we made the substitution $(w',z') = (w,z) Q_\theta\T$.  We observe the last line to be the desired result.

\end{proof}

\subsection{Example and the Inclusion of Noise}
What has been proven is that even with the offset Fourier coefficients, the partial sums still yield convolutions with similar kernels as in the symmetric case with the exception of an addition phase term.  Moreover, these results naturally extend to the 2D case with Fourier data on the polar grid, i.e. SAR data.  A visualization of this kernel is provided in Figure \ref{fig: kernel}.  For the parameters of the kernel, i.e. the sampling parameters $(K_c , M , \Delta k, \Theta)$, we repeat the setup from Figure 1 in \cite{SARMATLAB}.  Using our notation, these parameters are given by $M= 512$, angle of elevation $\phi = 30\degree$, central frequency $\alpha_c = \frac{\alpha_{512}+\alpha_1}{2} = 10$GHz, and bandwidth given by $ \alpha_{512}-\alpha_1 = 600$MHz.  The central angle is $\theta = 50\degree$, and the samples are acquired at 128 equispaced angles over a $3\degree$ total azimuth range.  The scene radius is chosen to be $R = 5m$, and the number of pixels is $N=500$, yielding a pixel size of 2cm.  Based on these parameters, the digital parameters may be determined based on the discussion in section \ref{sec:freqs}, and finally the kernel given by (\ref{main-2D-conv}) is constructed and visualized in Figure \ref{fig: kernel}.

In \cite{SARMATLAB} these imaging parameters were used to reconstruct 3 isotropic point scatterers (delta functions), and hence the resulting reconstruction according to Theorem \ref{main-theorem} is somewhat a superposition of three kernels $\mathcal K$.  Comparing their imaging result with the DB scale image of the kernel we have provided (bottom left) gives a visual confirmation of our derivation.

Finally, it is worth mentioning how these results differ with the inclusion of noise on the Fourier coefficients.  For simplicity consider the simple 1D Fourier sum as in (\ref{dirichlet-off}) and assume the Fourier coefficents there are replaced with noisy verisions, $\widetilde{\hat f_k} = \hat f_k + \epsilon_k$.  Then the result in (\ref{dirichlet-off}) takes the form
\begin{equation}
\begin{split}
 S_{K_1,K_2} f(x) 
 &= \sum_{k=K_1}^{K_2} \hat f_k e^{ikx} + \sum_{k=K_1}^{K_2} \epsilon_k e^{ikx} \\
 & = f*G(x;K_c , B) + \eta(x).
\end{split}
\end{equation}
Observe that the left sum is determined by Proposition \ref{prop:1}, and so this noisy version will still obtain the convolutions as before but with an added noise term $\eta(x)$ given by the right sum.  Suppose $\epsilon_k$ are i.i.d. mean zero complex Gaussian, e.g. $\epsilon_k = X + iY$, where $X$ and $Y$ are $N(0,\sigma^2)$.  Then since the sum of mean zero Gaussian random variables is a mean zero Gaussian random variable, we see that $\eta(x)$ is mean zero Gaussian, hence the added Gaussian noise in the coefficients results in added Gaussian noise in final image.  These arguments easily extend to our 2D convolution results from Theorem \ref{main-theorem}.

\begin{figure}
\centering
\includegraphics[width=1\textwidth,trim={1cm 0cm 1cm 0cm},clip]{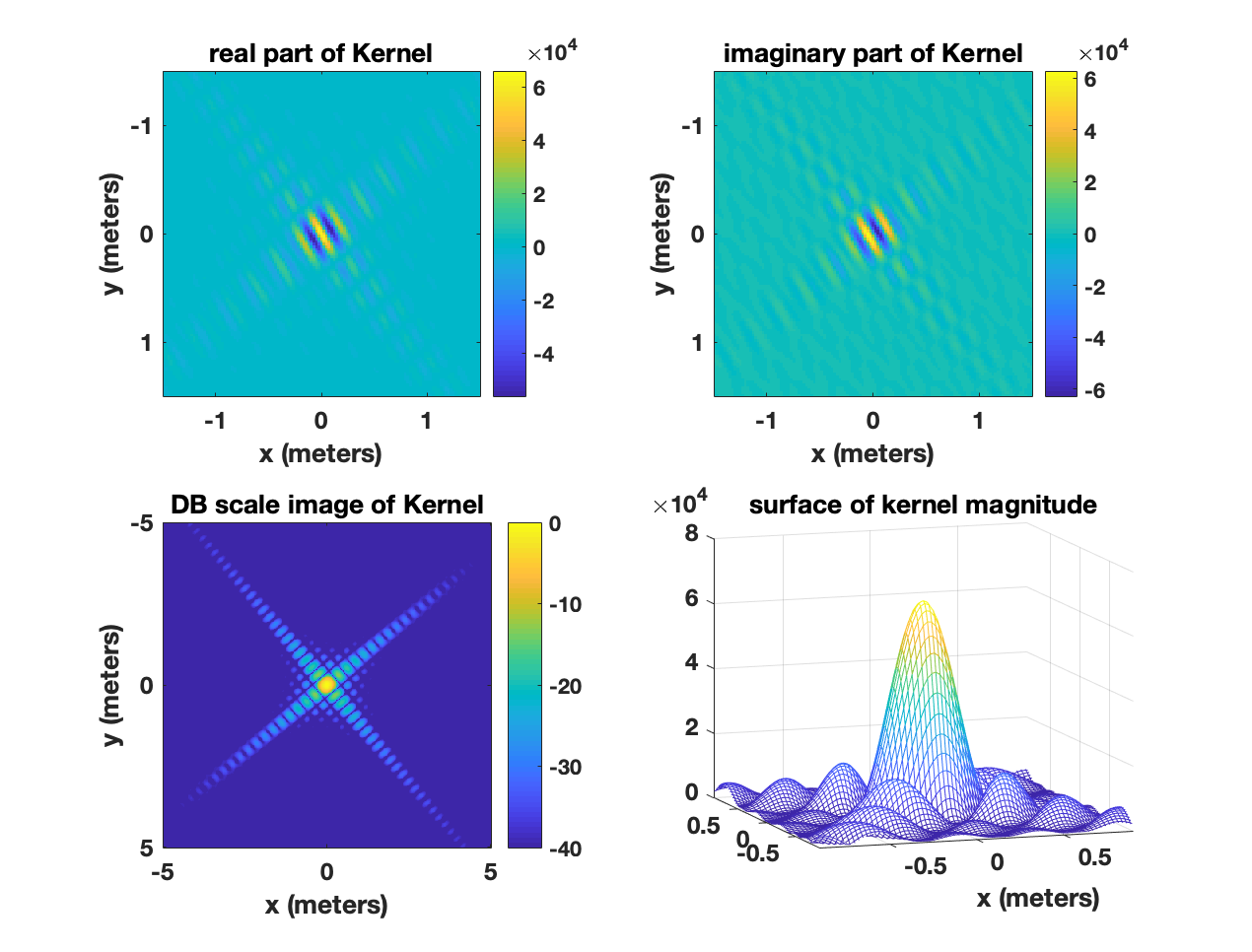}
\caption{Visualization of the kernel given by (\ref{main-2D-conv}) based on the imaging parameters in Figure 1 of \cite{SARMATLAB}.}
\label{fig: kernel}
\end{figure}

\section{Reconstruction of Complex Signals with Random Phases} \label{sec:rand}
It has been well documented that reconstructed images in SAR exhibit random phase values from one pixel to the next \cite{munson1984image,munson1984importance,munson86}.  This phenomenon occurs in a number of other imaging applications, and can be attributed to several factors.  In large part it is due to image digitization of an underlying continuous scene with complex valued scatterers that vary at a microscopic resolution far higher than we can expect to reconstruct (see for example, \cite{pinel2013electromagnetic}).  As it turns out, this property is simultaneously useful and destructive for reconstructing SAR images.

The most notable negative effect of the randomness is the infamous presence of speckle \cite{dainty2013laser,goodman2007speckle}, a multiplicative type of noise in the reconstructed images.  However, this complexity allows for high fidelity reconstructions from only a narrow band of high frequency Fourier data.  More specifically, for complex valued signals with uniformly distributed random phases, which we refer to as \emph{random complex signals}, it has been shown that the magnitude error in the reconstruction is dependent on the given bandwidth of data, $B$, and independent upon the band center, $K_c$.  This idea goes against conventional Fourier analysis where, due to the usual assumption of the decay of the coefficients for smooth functions, the most important frequency band with regard to reconstruction error is typically near the origin.

\subsection{Simple Examples with Random Complex Signals}
Before proceeding with the detailed discussion on these topics, a few numerical examples are first provided to motivate the mathematical derivations.  First we turn the attention towards Figure \ref{fig: phase1D} to emphasize some of these ideas first for one dimensional problems.  Shown in the figure are magnitudes of partial Fourier sums of the form (\ref{dirichlet-off}) of the 1D random complex signal with a magnitude given by
\begin{equation}
 |f(x)| =
 \begin{cases}
      0 ,& x \in [-\pi , 0 ) \\
      1, & x\in [0,\pi )
  \end{cases}.
\end{equation}
The setup in this problem is the same as in Figure \ref{fig: 1Dconv}, however random uniformly distributed phases were implemented onto the discretized signal $f$.  That is, $f_j = |f_j| e^{i\phi_j}$ where $\phi_j$ is a random variable with a uniform distribution on $[-\pi,\pi)$.  The bandwidth is again given by $B=50$ and the central frequencies used are $K_c = 0 \text{ and } K_c = 125$, i.e. one zero centered Fourier sum and one off-centered Fourier sum.  These partial Fourier sums are of course characterized by Proposition \ref{prop:1}.   Looking back at Figure \ref{fig: 1Dconv} with the real signal (no random phases added), the case of the zero centered sum results in an accurate representation of $f$, while the off-centered sum only captures the jump at $0$.  On the other hand, in the random phased case shown in Figure \ref{fig: phase1D}, the resulting reconstructions are visually equivalent, irrespective of the central frequency.  On the other hand they both appear quite noisy, which is again described as speckle noise.  One interpretation of this result is that due to the random phases on the signal, practically every nonzero point becomes an edge.  Therefore, although the offset Fourier data typically only captures edge information as indicated in Figure \ref{fig: 1Dconv}, edges for random phased signals are \emph{everywhere}.

The same observation can be made with off-centered 2D partial Fourier sums of the Shepp-Logan image shown in of Figure \ref{fig: phase}, where again one sum was taken for a real image and the other with added random phases.  The $k$ values used are akin to the polar sampling acquisition in SAR and are plotted in the bottom right panel, and hence the partial Fourier sum is equivalent to the convolution with the kernel $\mathcal K$ derived in Theorem \ref{main-theorem}.  Notice that with the real image only the edges are recovered, but with the random phased image we recovered the full image, albeit with the effect of speckle.  The expected value of the image (with the random variables again being the phases) is shown in the bottom left, which is explained more clearly in what follows.

\begin{figure}
 \centering
 \includegraphics[width=.45\textwidth]{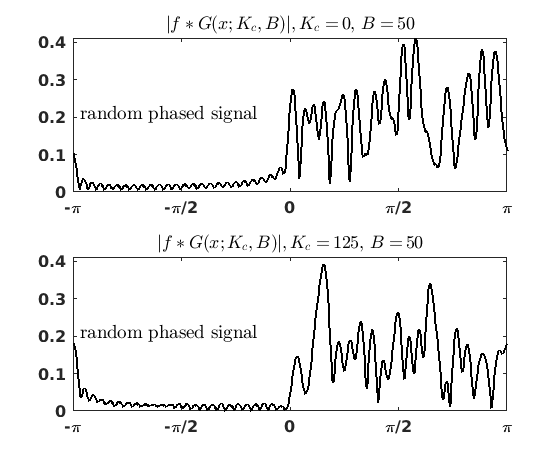}
 \caption{Repeating the partial Fourier sums from Figure \ref{fig: 1Dconv} with a random complex signal.  Observe the presence of speckle in the reconstructions and the apparent independence of the reconstruction and the central frequency, $K_c$.}
 \label{fig: phase1D}
\end{figure}

\begin{figure}
 \centering
 \includegraphics[width=1\textwidth]{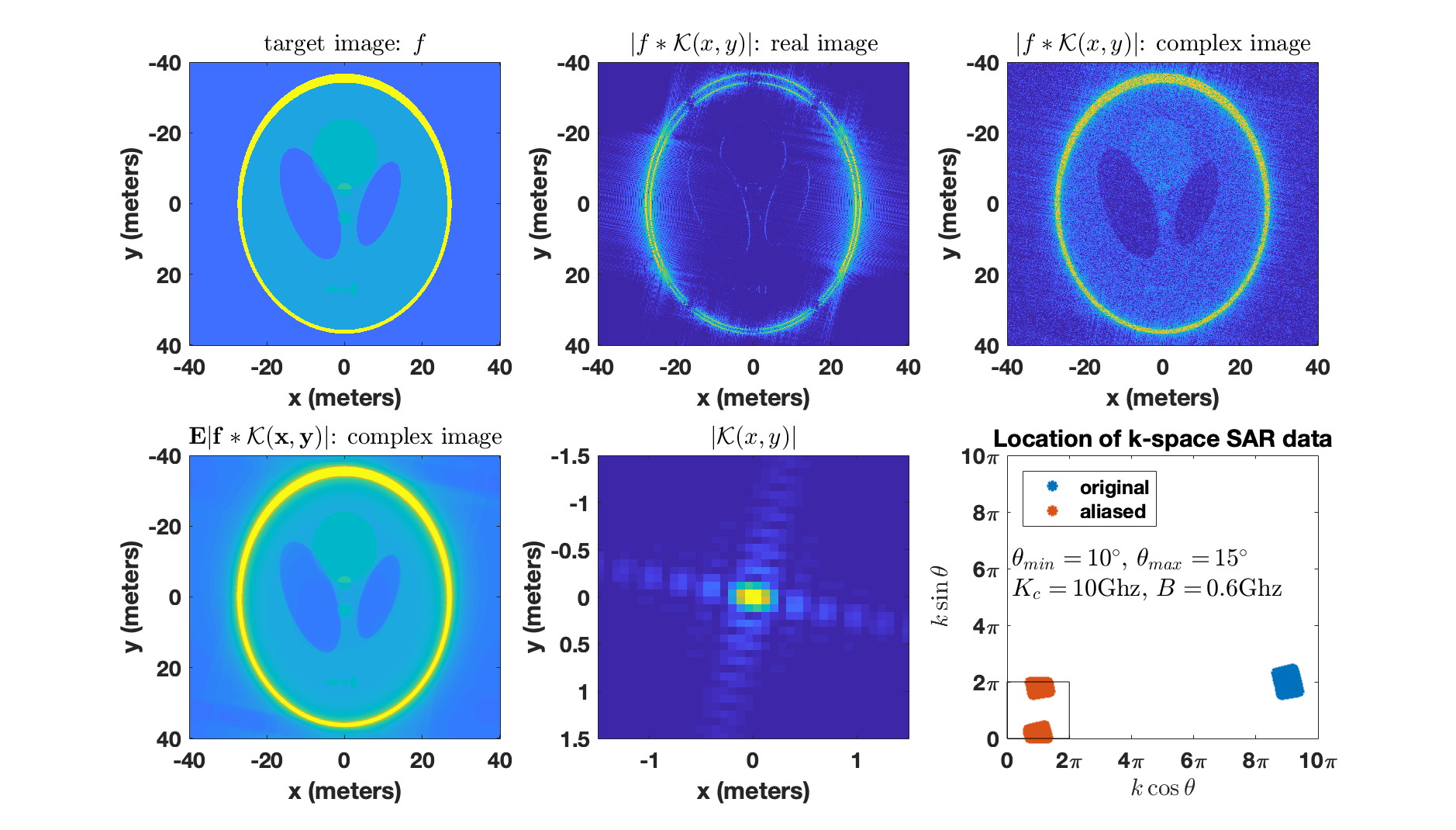}
 \caption{Partial Fourier sums of a 2D function $f$ with polar sampled data points, or equivalently by Theorem \ref{main-theorem}, convolutions of $f$ with $\mathcal K$. Observe the difference between the case when $f$ is real and when $f$ is a random complex signal.}
 \label{fig: phase}
\end{figure}

\subsection{Formal Characterization of Fourier sums of Random Complex Signals}

We proceed more formally with work related to \cite{munson1984image}. Consider $f\in \C^N$, where the phases of $f$ are independent uniform random variables.  Then we can write $f_j = e^{i\phi_j} |f_j|$, where $\phi_j$ are uniformly distributed over the interval $[-\pi , \pi)$.  Then the $k^{th}$ discrete Fourier coefficient is given by
\begin{equation}\label{fcoef1}
\hat f_k = \sum_{j=0}^{N-1} |f_j| e^{i\phi_j} e^{-i\frac{2\pi}{N}  k j}
\end{equation}
We may evaluate the magnitude of this coefficient as
\begin{equation}\label{fcoef2}
|\hat f_k|^2 = \sum_{j_1=0}^{N-1} \sum_{j_2=0}^{N-1} |f_{j_1}||f_{j_2}| e^{i(\phi_{j_1}-\phi_{j_2})} e^{-i\frac{2\pi}{N}k (j_1 - j_2)} .
\end{equation}
Under the assumption that $\phi_j$ are uniformly distributed, it is easy to see that the expected value of $e^{i(\phi_{j_1}-\phi_{j_2})}$ is given by
\begin{equation}\label{fcoef3}
\mathbb E\left(e^{i(\phi_{j_1}-\phi_{j_2})}\right) = 
\begin{cases}
0 & \mbox{if } j_1 \neq j_2\\
1 & \mbox{if } j_1 = j_2
\end{cases},
\end{equation}
which leads to 
\begin{align*}
\mathbb E \left(|\hat f_k|^2 \right) 
& = \sum_{j_1=0}^{N-1} \sum_{j_2=0}^{N-1} |f_{j_1}||f_{j_2}| \mathbb E \left(e^{i(\phi_{j_1}-\phi_{j_2})} \right) e^{-i\frac{2\pi}{N}k (j_1 - j_2)} \\
& = \sum_{j=0}^{N-1} |f_j|^2 = ||f ||_2^2.
\end{align*}
Thus the expectation of magnitude of the coefficients is independent of the frequency and they do not decay, which can be easily tested numerically.  This concept is contrary to conventional understanding of the behavior of Fourier transforms of smooth functions, and indeed we may consider signals with such random phase values as very "rough."

To expand upon this further, let us consider again the result from Proposition \ref{prop:1}.  In proving (\ref{dirichlet-off-2}), we used the continuous definition of $\hat f_k$.  However, even if we define $\hat f_k$ using a DFT, i.e.
\begin{equation}\label{DFT-def}
 \hat f_k = \sum_{j=0}^{N-1} f_j e^{-i\frac{2 \pi}{N} jk},
\end{equation}
then we would arrive at essentially the same conclusion, where the convolution $*$ is instead the discrete convolution.  More precisely, minor modifications to the proof show that
\begin{equation}\label{IDFT}
 S_{K_1,K_2} f_m = N^{-1} \sum_{k=K_1}^{K_2} \hat f_k e^{i\frac{2\pi}{N} mk}, \quad \text{for} ~ m=0,1,\dots, N-1,
\end{equation}
and 
\begin{equation}\label{dirchlet-off-4}
  S_{K_1,K_2}f = f*g, ~ \text{where} ~ g_m = N^{-1} e^{i\frac{2\pi}{N} K_c m} D_{B/2} \left(\frac{2\pi}{N} m\right),
\end{equation}
with $K_c$ and $B$ as in Proposition \ref{prop:1}.
This in essence is a simplified calculation needed to show the main result from Munson's original paper \cite{munson1984image}.  From here, writing the $N$ point discretized Dirichlet kernel as $D_{B/2}^N(x) = D_{B/2}(2\pi x /N)$ leads to
\begin{align*}
\left| S_{K_1,K_2} f_m \right|^2 
& = N^{-2} \left| \sum_{j=0}^N |f_j| e^{i\phi_j} e^{i\frac{2\pi}{N}K_c (m-j)} D_{B/2}^N (m-j) \right|^2 \\
& = N^{-2} \sum_{j_1,j_2=0}^{N-1} |f_{j_1}||f_{j_2}| e^{i(\phi_{j_1}-\phi_{j_2})}e^{i\frac{2\pi}{N}K_c (j_1-j_2)} D_{B/2}^N(m-j_1) D_{B/2}^N(m-j_2) 
\end{align*}
Taking the expected value of this expression and using (\ref{fcoef3}) leads to
\begin{equation}
\mathbb E\left( \left| S_{K_1,K_2} f_m \right|^2 \right) 
= N^{-2} \sum_{j=0}^{N-1} |f_j|^2 \left(D_{B/2}^N(m-j) \right)^2 = N^{-2} \left( |f|^2 * |D_{B/2}^N|^2\right)_m .
\end{equation} 
Notice once again this expression depends only on the bandwidth $B= K_2 - K_1$ and is independent of the central frequency $K_c= (K_1+K_2)/2$.  Hence the random phases are somewhat "helpful" in providing some resolution even when we only have high frequency information.  A formal summary of what we have proven is given below.

\begin{theorem}
 Let $f\in \C^N$, $K_1 , K_2 \in \Z$, and let $\phi_j = \arg (f_j)$ be independent uniformly distributed random variables on the interval $[-\pi , \pi)$.  Define $K_c = (K_1 + K_2)/2$ and $B = K_2 - K_1$, and consider the partial Fourier sum $S_{K_1 , K_2} f$ as in (\ref{IDFT}).  Then the expected value of the squared magnitude the partial Fourier sum is dependent only on the bandwidth $B$ and is given by
 \begin{equation}
 \mathbb E\left( \left| S_{K_1,K_2} f_m \right|^2 \right) 
= N^{-2} \left( |f|^2 * |D_{B/2}^N|^2\right)_m,
 \end{equation}
 for $m= 0 , 1 , \dots , N-1$.

\end{theorem}

\subsection{Heuristic Extension to Continuous Setting}
We showed in the previous section that assuming random phases on a discrete signal results in the Fourier coefficients with expected magnitudes that are independent of the frequencies.  These ideas do not seamlessly extend to continuous signals (which is what is inevitably sampled in the data), especially since we cannot assume the phases of the underlying function are completely random from one point in the domain to the next (such a function would not be even Lebesgue integrable).  As an alternative, we assume that the phase values of the function vary so rapidly (e.g. at a microscopic level \cite{pinel2013electromagnetic}) that from one pixel to the next the phases are essentially random.  

With this in mind, we write a 1D complex signal $f$ as $f(x) = |f(x)| e^{i\phi(x)}$, where $\phi(x)$ is again a random variable.  In this case though, we assume there is some covariance function, $\text{cov}(\phi(x),\phi(y)) = \Phi(|x-y|)$, which decays possibly very rapidly.  Then using (\ref{dirichlet-off-2}) we compute the magnitude of the offset partial Fourier sum as
\begin{equation}\label{eq:cont-case}
|S_{K_1 , K_2 } f(x)|^2 = \int_{-\pi}^\pi \int_{-\pi}^\pi |f(y_1)| |f(y_2)| e^{i(\phi(y_1) - \phi(y_2))} G_{K_c , B} (x - y_1) \overline{G_{K_c , B}(x - y_2)} \, dy_1 \, dy_2
\end{equation}

To proceed, consider a very simple case where the covariance relationship is approximated by
\begin{equation}
\mathbb E \left[ e^{i(\phi(x) - \phi(y))}\right] \approx \begin{cases}
         1, & |x-y|<\delta\\
         0, & \text{otherwise}
        \end{cases},
\end{equation}
for some very small $\delta$.  Then taking the expected value of (\ref{eq:cont-case}) leads to
\begin{equation}
\mathbb E\left[ |S_{K_1 , K_2 } f(x)|^2 \right] \approx \int_{-\pi}^\pi \int_{y_2 - \delta}^{y_2 + \delta} |f(y_1)| |f(y_2)| G_{K_c , B} (x - y_1) \overline{G_{K_c , B}(x - y_2)} \, dy_1 \, dy_2.
\end{equation}
Finally, if $\delta$ is small enough, we suppose that for $y_1 \in (y_2 - \delta, y_2 + \delta)$ we have $|f(y_1)| \approx |f(y_2)| $ and $G_{K_c,B} (y_1) \approx G_{K_c , B} (y_2)$,\footnote{Observe that this assumption is okay so long as the central frequency $K_c$ is not too large relative to $\delta$.} hence
\begin{equation}
\mathbb E\left[  |S_{K_1 , K_2 } f(x)|^2 \right] \approx 2 \delta \int_{-\pi}^\pi |f(y)|^2 |G_{K_c, B}(x-y)|^2 \, dy = 2\delta |f|^2 * |G_{K_c , B}|^2 (x) = 2\delta |f|^2 * |D_{B/2}|^2 (x).
\end{equation}
Therefore we conclude, with our assumptions, that in continuous case the resulting reconstruction from partial Fourier sums of random complex signals depends primarily on the bandwidth, independently of the central frequency.  These exact arguments are extended naturally to 2D partial Fourier sums resulting from SAR sampling geometries with the aid of Theorem \ref{main-theorem}.

\section{Regularization}\label{sec: reg}
The partial Fourier sums as presented in Section \ref{sec:filter} are usually the most computationally convenient approach to yield quick approximations of functions from such Fourier data.  Inherent to these approximations is the assumption that the Fourier coefficients outside of the known band of data are zero, which may result in ringing or Gibbs artifacts.  Filtering methods, as briefly discussed at the beginning of section \ref{sec:filter}, is an inexpensive way to alleviate some of these unwanted artifacts simply by weighting or smoothing the coefficients in the Fourier sum in a particular way.  However, this weighting inevitably results in a loss of resolution.  This should be clear, since the filtering effectively changes the values of the known coefficients to incorrect values with the introduction of these weights.

Alternatively, one can work around the inherent zeroing of the unknown coefficients present in partial sums by introducing a numerical method to ``fill in'' for the missing coefficients in a natural way.  This can be done by implementing regularizations such as total variation (TV) that yields an \emph{optimal} solution given the known coefficients, making no assumptions on the unknown coefficients. In other words, finding a \emph{best} solution over the set of all possible solutions that satisfy the known coefficients.  In doing this we do not lose any information in the coefficients as done with filtering, while still opening the possibility to reduce speckle and ringing artifacts.  These methods are described in detail in the proceeding section.

Before doing so, we illustrate these ideas in Figure \ref{fig: 1D}, where the periodic function
\begin{equation}
 f(x) =
 \begin{cases}
      2x + 2 ,& x \in [-1/2,0) \\
      2x, & x\in [0, 1/2).
  \end{cases},
\end{equation}
is reconstructed from its truncated Fourier coefficients.  These coefficients are first computed analytically as 
$$
\hat f_k = \int_{-1/2}^{1/2} f(x) e^{-i2\pi kx} \, \rmd x = \frac{i}{\pi k},
$$
for $k\neq 0$.  For the reconstruction we used the values $|k|\le 75$.

Clearly, this partial sum results in the Gibbs effect at the discontinuity, and the filtered sum smooths this region over.  However, the TV regularized solution removes the Gibbs ringing without over-smoothing the solution.  Another explanation is again given by observing the magnitude of the Fourier coefficients in the right panel.  Obviously the partial sums and filtering results in coefficients of zero outside the known band.  Using the TV regularization only enforced the solution to agree on the known coefficients, and the unknown coefficients from the reconstructed solution (green curve) are similar to the original coefficients. 
\begin{figure}
 \centering
 \includegraphics[width=.8\textwidth]{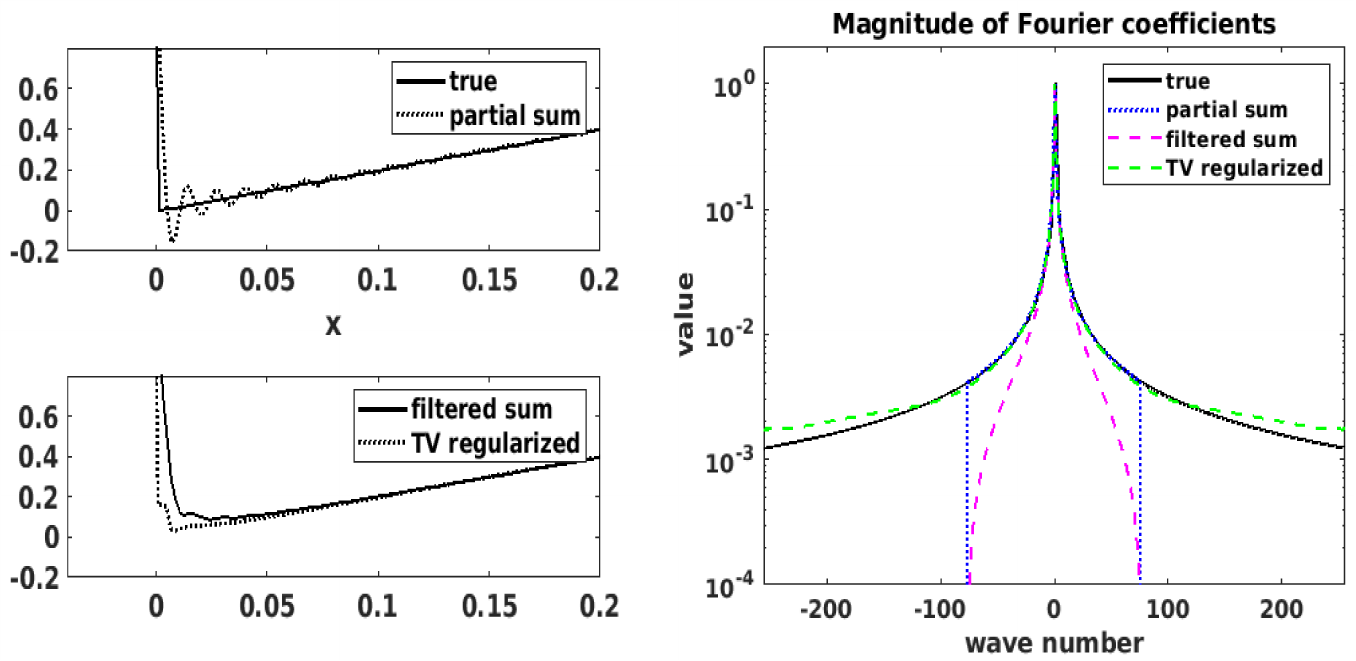}
 \caption{A one-dimensional example of reconstruction from partial Fourier data.}
 \label{fig: 1D}
\end{figure}

\subsection{General Description of Regularization for SAR}
Consider the reconstruction of a discretized image or signal $\bff$ from some Fourier or SAR data set denoted by $\hat \bff= \{ \hat f_{k,\theta } \}_{k,\theta}$ .  Let $\mathbf F$ be the DFT operator that maps $\bff$ onto $\hat \bff$, where the digital frequencies needed for $\mathbf F$ may be determined from the SAR acquisition geometry and pixel size as described in section \ref{sec:freqs}.  Moreover, assume that matrix-vector multiplications with the operator $\mathbf F$ may be computed efficiently by say a nonuniform FFT \cite{Greengard, Andersson, 1166689}, opening the door for the possiblity of iterative reconstruction methods for $\bff$\cite{sanders2017composite}. Also note that for an even more efficient approach, the data could instead be first regridded onto the uniform mesh in Fourier space generating a new data vector so that simply FFTs can be applied.  A regridding technique will use the same concepts that the nonuniform FFT methods use \footnote{For details on the implementation of these ideas for SAR, see for example \cite{sanders2017composite} as well as the code associated with this document.}.

With all of this in mind, the general ideal inverse problem that we are interested in solving, which possibly avoids zeroing of the unknown Fourier values, is written by
\begin{equation}\label{eq: reg}
\bff^* = \arg \min_{\bff} \Big\{  \| \mathbf F  \bff - \hat \bff \|_2^2 + \lambda H(\bff) \Big\}  .
\end{equation}
Here, $H$ is a regularization or prior assumption that incorporates prior knowledge about the behavior of $\bff$ to encourage \emph{favorable} or smooth solutions, and $\lambda$ is an important parameter to balance the data fitting and regularization.  A classical choice is the Tikhonov regularizer of the form $H(\bff ) = \| \bff \|_2^2$.  Many variants of Tikhonov regularization exist, for instance we may first take the discrete derivative of $\bff$ for a regularization term $H(\bff ) =\| \mathbf D\bff \|_2^2$, where
\begin{equation}\label{eq: Dmatrix}
\mathbf D = 
\left( 
 \begin{array}{ccccc}
  -1 & 1 & 0 & \dots & 0\\
  0 & -1 & 1 & \dots & 0 \\
  0 & 0 & -1 & \dots & 0 \\
  \vdots & \vdots & \vdots & \ddots & \vdots \\
  0 & 0 & \dots & -1 & 1
 \end{array}
 \right) .
\end{equation}
With this type of regularization the minimizer to (\ref{eq: reg}) has an analytical solution given by
\begin{equation}\label{eq: Tik}
 \bff^* = (\overline{\mathbf F}\T \mathbf F + \lambda \mathbf D\T \mathbf D)^{-1} \overline{\mathbf F}\T \hat \bff.
\end{equation}

More recently $\ell_1$ regularizers have been used, such as TV regularization.  For TV regularization in 1D the regularization term $H$ is given by
\begin{equation}
TV(\bff ) = \| \mathbf D\bff \|_1 = \sum_{j=1}^{N-1}| \bff_{j+1} - \bff_j | .
\end{equation}
The TV norm and Tikhonov regularizations may be naturally extended to problems with dimensions higher than 1D \cite{bregman}, but for simplicity we will continue our descriptions in the 1D case.

For SAR, a small total variation prior is actually not suitable as suggested by the content in section \ref{sec:rand}.  In particular, we can presume that the magnitude of the underlying scene has a small total variation, but not the complex values due to the inherent randomly varying phases on the pixels.  Therefore a more appropriate regularization prior may be $TV (| \bff |)$.  However, this cannot be implemented efficiently since the absolute value as an operator is nonlinear, and typical $\ell_1$ optimization algorithms are designed for linear operators inside of the $\ell_1$ norms.

As an alternative it has been shown effective to instead first determine an approximate solution $\bff^0$ \cite{sanders2017composite}.  Using this solution, we define a diagonal matrix containing the phases of $\bff^0$ by $\Theta = \text{diag}\{ e^{-i \arg\left( \bff_{j}^0 \right) }\}_{j=1}^N$, and it follows that $| \bff^0 | = \Theta \bff^0 $.  Hence a suitable regularization involving linear operators is given by 
\begin{equation}
TV(\Theta \bff ) =\| \mathbf D \Theta \bff \|_1 = \sum_{j=1}^{N-1} \left| e^{-i \arg\left( \bff_{j+1}^0 \right)} \bff_{j+1} - e^{-i \arg\left( \bff_{j}^0 \right) } \bff_j \right| \approx \sum_{j=1}^{N-1} \left| \, |\bff_{j+1}|  - | \bff_j | \, \right| 
\end{equation} 

Likewise, any general $\ell_1$ optimization with a linear operator $T$ can be adapted to the complex problem by using the regularization $\| T \Theta \bff \|_1$.  There is a mild difficulty associated with solving the optimization problem (\ref{eq: reg}) when $H$ is an $\ell_1$ prior due to the non differentiability of the $\ell_1$ norm.  However, a number of efficient methods have been developed to do so in recent years.  An overview of some of the main concepts is provided in the appendix, and the interested reader will find a sufficient literature devoted to this topic.  The most popular perhaps are the split Bregman method and the alternating direction method of multipliers (ADMM) \cite{bregman,Li2013}, and implementation of these tools within the SAR framework is given in \cite{sanders2017composite}.  

Finally we mention that regularization methods have been very well studied resulting in a huge number approaches and variations \cite{achim2003sar,bredies2010total}, as well as many interpretations \cite{steidl2004equivalence,kaipio2006statistical}.  For example, there are higher orders TV methods, wavelets, and k-space analysis of the methods, some of which is summarized and expanded on in recent work\cite{sanders2018multiscale}.

A final example on openly available SAR data sets is given in Figure \ref{fig: final-reg}, where we compare a simple inverse NUFFT solution to regularized solutions.  On the left is the GOTCHA data set \cite{Gotcha}, where TV regularization was used. On the right is the CV data \cite{dungan2010civilian}, also shown in Figure \ref{fig: example}.  In this case the regularization was just the $\ell_1$ norm of the image, $\| \bff \|_1$.  The NUFFT solutions clearly exhibit speckle and ringing artifacts, and these effects are significantly reduced in the regularized solutions, without degrading the true image features.  In the bottom row the corresponding logarithmic scale of the 2D FFTs of the images are shown.  The Fourier transforms of the NUFFT solutions clearly show us the location of the aliased Fourier data, and the dark regions on these images indicate where no data was located.  On the other hand, the Fourier transforms of the regularized solutions show us that in some sense are able to fill in these k-space regions void of data.
\begin{figure}
 \centering
 \includegraphics[width=1\textwidth]{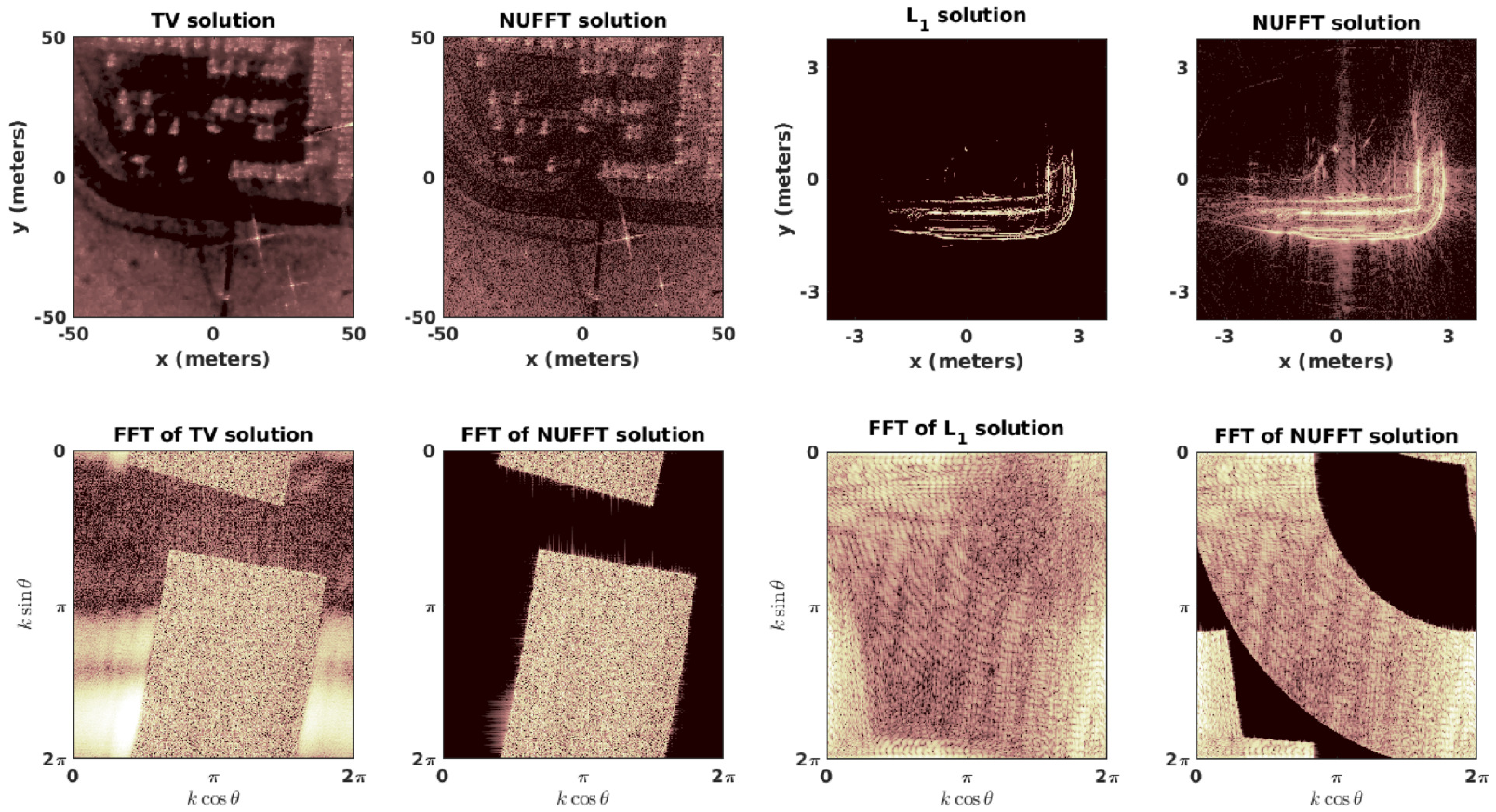}
 \caption{Comparison of regularized solutions with a standard inverse NUFFT solutions.  On the left is the GOTCHA data \cite{Gotcha} and the right is the CV data \cite{dungan2010civilian}.  The bottom row shows the corresponding Fourier transforms of the solutions, and indicate that the regularizations are able to fill in these regions of missing data in Fourier space.}
 \label{fig: final-reg}
\end{figure}

\section*{Acknowledgments}
This work is supported in part by the grants NSF-DMS 1502640 and AFOSR FA9550-15-1-0152.

\appendix
\section{Methods for Efficient $\ell_1$ Optimization}
Here we consider numerical solutions to the general $\ell_1$ optimization problem to 
\begin{equation}\label{eq: main-L1}
f^* = \arg\min_{f\in \C^N} \| A f- b \|_2^2 + \lambda \| T f\|_1,
\end{equation}
which is solved to obtain regularized solutions $f^*$ for imaging problems.  We let $T \in \C^{K\times N}$, $A \in \C^{M\times N}$, and $b\in \C^M$.  For SAR, the matrix $A$ is a discretized Fourier operator, $b$ is the SAR echo data, and $T = D\Theta$, where $\Theta$ is a diagonal matrix constructed from the phase estimates of the pixels of $f^*$ and $D$ is say a finite difference matrix, all of which is described in the main text.

The inherent difficulty in this optimization problem is due to several factors: primarily the large scale of the dimensions $M$ and $N$ for two and three-dimensional imaging problems combined with the non-differentiability of the $\ell_1$ norm.  In this section we outline the popularized method known as the alternating direction method of multipliers (ADMM) for solving (\ref{eq: main-L1}).  For more detailed careful convergence analysis of these and related methods, see for instance \cite{Li2013,wu2010augmented}.  Of particular note provided here however, is the extension of some of these ideas for complex valued signals.  The MATLAB software for the optimization methods described below is openly available \cite{toby-web, sanders2018multiscale}.


First, let us rewrite the optimization problem in an equivalent form
\begin{equation}\label{eq: l1c}
  \min_{f,g} \| Af - b \|_2^2 + \lambda \|g\|_1 \quad s.t. \quad Tf = g.
\end{equation}
We can consider (\ref{eq: l1c}) as a constrained optimization problem.  Therefore we may convert it to an unconstrained optimization problem that approximates the constrained problem using an augmented Lagrangian functional \cite{hestenes1969multiplier,Li2013}.  This functional takes the form
\begin{equation}\label{eq: l1ag}
 \mathcal L (f,g ,\sigma ) = \| Af-b\|_2^2 + \lambda \| g \|_1 + \frac{\beta}{2} \|Tf - g \|_2^2 - \sigma^H (Tf - g),
\end{equation}
where $\sigma \in \C^K$ is a Lagrangian multiplier and the superscript $H$ will be used to denote the conjugate transpose.  If $\sigma$ is updated (or chosen) appropriately, as well as $\beta$, then the minimizer to (\ref{eq: l1c}) becomes a local minimum to (\ref{eq: l1ag}) \cite{hestenes1969multiplier}.  In other words, an ordinary Lagrange multiplier method converts a constrained optimization problem into an equivalent unconstrained problem of finding a saddle point, the augmented Lagrangian functional converts this saddle point to a local minimizer.  Hence, gradient descent types of approaches may be used to find this minimizer.  While it may appear to the reader that this reformulation has made the minimization problem more complicated, the introduction of the splitting variable $g$ makes the problem much simpler to solve.  The general approach to minimizing $\mathcal L$ is by an alternating minimization over $f$, $g$, and $\sigma$ until convergence.

A basic outline for the iterations is given by the following, where $k$ denotes the iteration:
\begin{equation}\label{eq: iters}
 \begin{split}
  f^{k+1} &= \arg \min_f \mathcal L(f,g^k , \sigma^k) \\
  g^{k+1} & = \arg \min_g \mathcal L (f^{k+1},g,\sigma^k) \\
  \sigma^{k+1} & = \sigma^k - \beta (Tf^{k+1} - g^{k+1}) .
 \end{split}
\end{equation}
The updates on $\sigma$ are standard for augmented Lagrangian methods \cite{hestenes1969multiplier}.  The updates on $g$ are given by a simple formula described below in section \ref{sec: g}.  The updates on $f$ are typically given by a crude approximation that may be computed in a pragmatic way using a single gradient decent.  The reason for this is because computing the exact minimizer (within some tolerance) may require many inner loop iterations, and empirical evidence shows that a single step gradient approximation to be suitable for the alternating approach given by (\ref{eq: iters}), i.e. we only need to roughly solve for $f^{k+1}$.  A basic description for computing solutions to each subproblem in (\ref{eq: iters}) are given in the following subsections.

\begin{figure}
\centering
\includegraphics[trim={6cm 0cm 4cm 0cm},clip,width=1\textwidth]{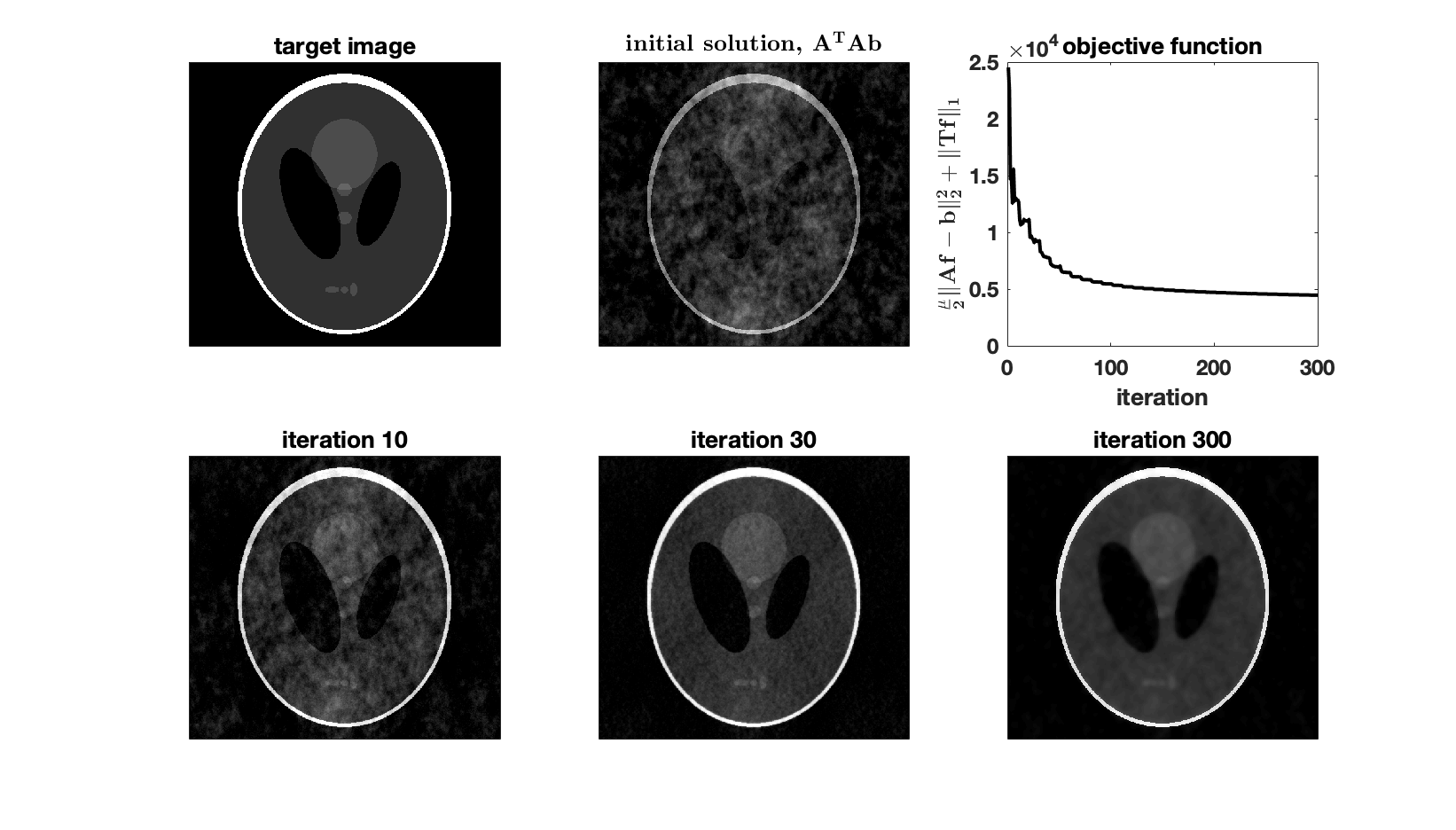}
\caption{A 2D example empirically demonstrating the convergence of the ADMM algorithm outlined by (\ref{eq: iters}).}
\label{fig: 2d}
\end{figure}
Before outlining these details, a simple example demonstrating the convergence of the algorithm described is presented Figure \ref{fig: 2d}.  The sampling matrix in this case takes the form $A = P\mathcal F$, where $\mathcal F$ is the unitary discrete Fourier transform and $P$ is a row selector matrix that randomly selects half the rows of $\mathcal F$.  In other words, $A$ is a partial Fourier transform.  The noise added to $b$ is mean zero i.i.d. complex Gaussian with an SNR of 2.  The figure shows the solution at different iterations, and a plot showing the value of the objective function (\ref{eq: main-L1}) as a function of the iteration.  For the most part, we observe that the algorithm iteratively yields solutions so that the objective function is monotone decreasing.


\subsection{Solution to the $g$ subproblem} \label{sec: g}
  The main benefit of the reformulation of (\ref{eq: main-L1}) as (\ref{eq: l1ag}) is that at any iteration, the minimizer to $g$ for any given set of all other variables is given by the exact shrinkage formula \cite{bregman,Li2013}
\begin{equation}\label{eq: updateg}
 g^{k+1} = \max \left(| Tf^{k+1} - \sigma^k/\beta| - \lambda /\beta ,0 \right) * \text{sign} \left( Tf^{k+1} - \sigma^k/\beta\right) .
\end{equation}
Hence, the issue of the non-differentiability of the $\ell_1$ norm is circumvented with the splitting variable and resulting shrinkage formula.  For the real valued case, this shrinkage formula is well-known.  We are not aware of a formal proof of the shrinkage formula in the complex case.  We provide it here, as a direct consequence of the following proposition.

 \begin{proposition}
  Consider the functional 
  $$
  F(z) = \frac{
  \beta}{2} |z-z_0|^2 + |z| + \mathrm{Re}(\overline \sigma z)  ,
 $$
 for $z,z_0,\sigma\in\C$ and $\beta>0$.  Then the minimizer of $F$ is given by
 $$
 z = \max\left(|z_0-\sigma/\beta| - 1/\beta,0\right) * sign(z_0-\sigma/\beta).
$$
\end{proposition}
\begin{proof}
Write $\sigma = \sigma_1 + i\sigma_2$, $z = x+iy$, and $z_0 = x_0 + iy_0$.  Supposing $|z|\neq 0$, then the partial derivatives of $F$ are given by
\begin{equation}
\begin{split}
F_x(z) = \beta(x-x_0) + \sigma_1 + x/|z| \\
F_y(z) = \beta(y-y_0) + \sigma_2 + y/|z| .
\end{split}
\end{equation}
Setting these to zero and rearranging yields
\begin{equation}\label{xy-shrink}
\begin{split}
x & = \frac{|z|( x_0-\sigma_1/\beta)}{|z|+1/\beta}\\
y & = \frac{|z| (y_0-\sigma_2/\beta)}{|z|+1/\beta}.
\end{split}
\end{equation}
Squaring these two equations and adding together we find $|z|$ to be given by
\begin{equation}\label{z-shrink}
|z| = |z_0 - \sigma/\beta| - 1/\beta.
\end{equation}
Observe this only makes sense for $|z_0 - \sigma/\beta| \ge 1/\beta$, and indeed it can be deduced that in the alternate case the solution is $|z| = 0$.  Substituting (\ref{z-shrink}) into (\ref{xy-shrink}) and combining this with the case $|z| = 0$ completes the proof.
\end{proof}

\subsection{Approximation to the $f$ subproblem}
For the minimization over $f$, notice that the functional just includes linear and quadratic terms, therefore a gradient decent method can be efficiently implemented so long as the gradient can be computed efficiently.  The gradient of $\mathcal L$ with respect to $f$ may be derived as
\begin{equation}\label{eq: grad}
 \nabla_f \mathcal L (f , g , \sigma ) = 
 2  A^H (Af-b ) + \beta T^H ( Tf - g ) - T^H \sigma ,
\end{equation}
and therefore updates on $f$ take the form 
\begin{equation}\label{eq: updatef}
 f^{k+1} =  f^k - \tau^k \nabla_f \mathcal L (f^k , g , \sigma ) 
\end{equation}
for the current $g$ and $\sigma$, and where $\tau^k$ is some appropriate step length.  We leave the appropriate choice of $\tau$ as an exercise for the reader, though a spectral step length is known to be effective \cite{barzilai1988two}.  In \cite{li2010efficient} the spectral step is used along with backtracking.

\subsection{Computational Cost and Efficient matrix-vector multiplication}
Observe that the speed of the algorithm described essentially relies on efficient computation of (\ref{eq: grad}) and (\ref{eq: updateg}) at each iteration.  Therefore the main computational cost will be due to matrix-vector products, and it is critical to have efficient methods for matrix vector multiplication with $A$, $T$, and their adjoints.  

Suppose for simplicity that $A$ and $T$ are square operators and $f\in \C^N$, in which case a full matrix vector product cost is $O(N^2)$ (not to mention the memory needed to store such a matrix).  However, the operator $T$ and its adjoint can typically implemented in much less since it is usually a sparse operator with mostly zeros.  For example, a finite difference operation requires a cost of only $O(N)$.  Even wavelet operators, which as matrix operators are somewhat sparse, can be computed using fast wavelet transforms and FFTs in $O(N\log N)$ \cite{cody1992fast} .  

The operator $A$ on the other hand can generally be more cumbersome.  For SAR with where the data are non-equispaced Fourier coefficients and hence the matrix $A$ is a Fourier operator, a direct calculation of the matrix vector product with the full matrix would return us back to $O(N^2)$ iterations.  In addition, FFT operations cannot be implemented directly (typically an order $N\log N$ operation).  To save us from this computational burden however are nonuniform FFT (NUFFT) methods, which can be carried out in the same order of computing time as FFTs. For extensive details on these techniques and their developments see for instance \cite{1166689,Greengard,Averbuch2006145}, where \cite{1166689} also provides corresponding software.  A summary of some of these methods in the context of SAR is given in \cite{Andersson}.  

As an alternative to using NUFFTs at each iteration, one may first re-grid (or interpolate) the nonuniform data $b$ onto a uniform grid to generate a modified data vector.  This requires an accurate re-gridding method, which usually requires many of the same ideas as NUFFTs.  Once the re-gridding is carried out, then typical FFTs may be applied for $A$. Our preference is to re-grid the data first, since empirical evidence suggests that there is no loss in accuracy, but there is a slight improvement to the speed of the algorithm.  For doing so, we have implemented the software provided by Fessler et al \cite{1166689}.

\section{Criteria for Optimal $\ell_1$ Solutions}
In this section we outline some of the analytical convergence properties and characterizations of the solutions to the $\ell_1$ optimization problem.  For additional details along these lines, see for instance \cite{wu2010augmented,tibshirani2013lasso}

For the augmented Lagrangian functional (\ref{eq: l1ag}), the necessary first order optimality conditions to guarantee a critical point are given by
\begin{equation*}
\begin{split}
& \nabla_f L(f,g,\sigma) = 0\\
& \nabla_g L(f,g,\sigma) \ni 0\\
& \nabla_\sigma (L(f,g,\sigma) = 0.
\end{split}
\end{equation*}
The second condition is set valued since the sub-differential is required due to the $\ell_1$ term, which in 1D is given by
$$
\frac{d}{dx}|x| = sign^*(x)=
 \begin{cases}
      -1 ,& x <0 \\
      1, & x>0\\
      [-1,1] & x=0
  \end{cases},
$$
where we are using $sign^*$ to denote the set valued sign function. Computing the last condition gives $Tf = g$, which simply implies the solution satisfies the original problem constraint in (\ref{eq: l1c}).  Computing the first two conditions leads to
\begin{equation*}
\begin{split}
& \mu A^H (Af-b) + \beta T^H (Tf-g) - T^H \sigma = 0\\
& sign^*(g) + \beta (g-Tf) + \sigma \ni 0.
\end{split}
\end{equation*} 
Observe further that this second condition implies
$$
\beta (g-Tf) + \sigma \in -sign^*(g)
$$
These two conditions are checked for a simple test problem shown in the top row of Figure \ref{fig: opt}, where the $\ell_1$ optimization algorithm outlined in section 1 was evaluated for 5000 iterations. For this test problem, the matrix $A \in \R^{400\times 500}$, $T$ is a second order finite difference operator, and the test signal is a sine curve.  The noise added to $b$ is mean zero i.i.d. Gaussian with an SNR of 5.  Observe that the two conditions are approximately satisfied.

Returning to the original problem (\ref{eq: main-L1}), the first order necessary and sufficient condition to guarantee a given solution is a minimizer may be computed as
\begin{equation}\label{optim-1}
\mu A^H (Af-b) + T^H sign^*(Tf) \ni 0.
\end{equation}
To explore this condition further, let $S = \{ j \, | \, (Tf)_j\neq 0\}$, $R=S^c$, and $T_S$ denote $T$ containing only the rows from $S$ (similarly for $R$).  Then 
$$
T^H sign^* (Tf) = T_S^H sign(T_S f) + T_{R}^H sign^*(T_R f).
$$
Substituting this into (\ref{optim-1}) and rearranging (assuming $R \neq \varnothing$) yields
\begin{equation}
\left| (T_{R}^H)^+ \left(\mu A^H(Af-b) + T_S^H sign(T_Sf)   \right) \right| \le 1
\end{equation}
This condition can be more cumbersome to check in general, since we need to solve a least squares problem, $T_R^H x = y$, for $x$.  In some cases this pseudo-inverse can be computed analytically however, e.g. if $T$ is a circulant finite difference matrix then one can take advantage of the matrix Fourier diagonalization.  This condition is checked in the bottom left of Figure \ref{fig: opt} for the same test problem used to check the Lagrangian optimality conditions.  For computational reasons, the set $S$ was determined by $S  = \{ j \, |\, |(Tf)_j| >10^{-3} \}$.
\begin{figure}
\centering
\includegraphics[width=1\textwidth]{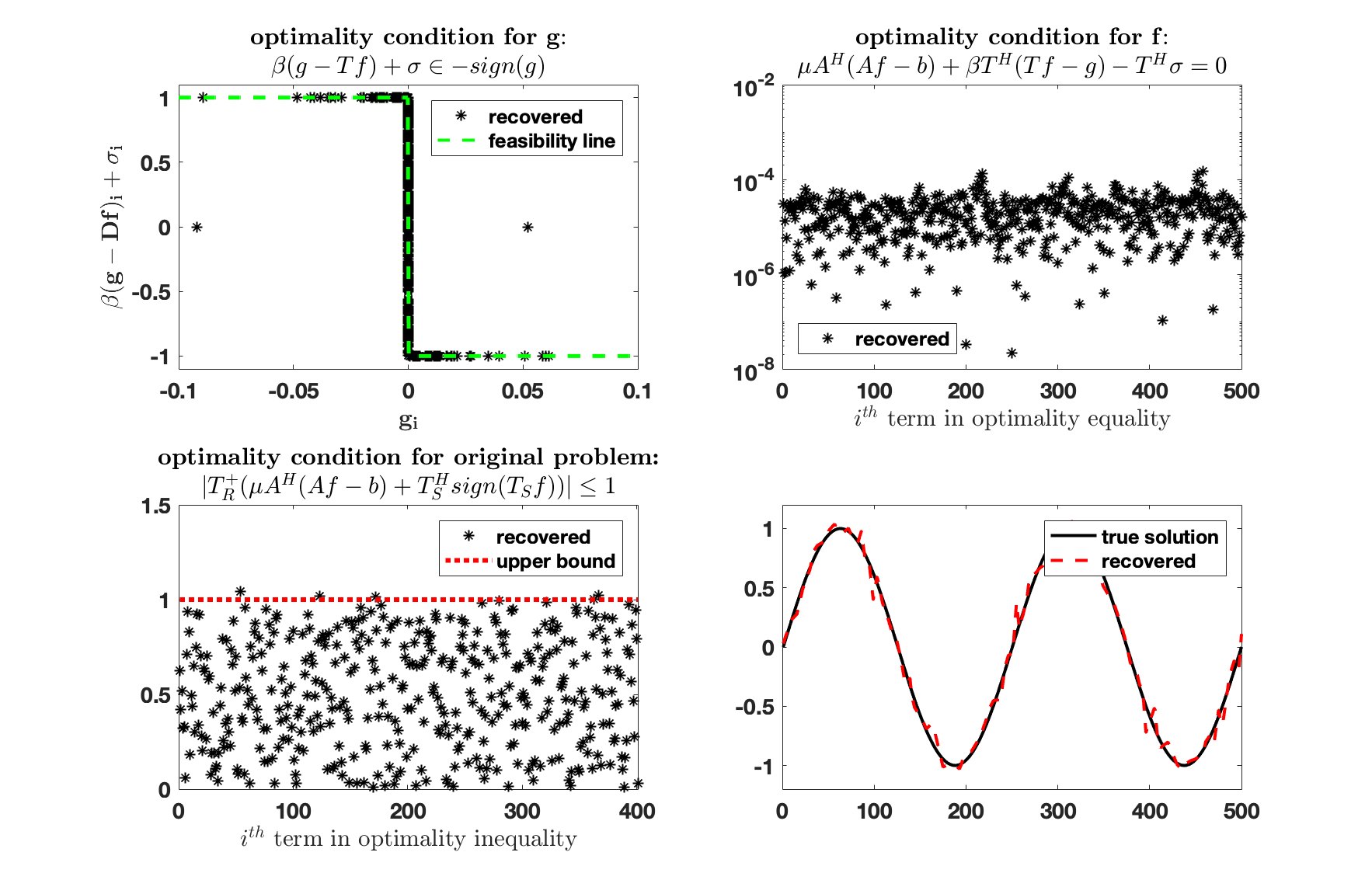}
\caption{Optimality conditions for the augmented Lagrangian function (\ref{eq: l1ag}) is shown in the top.  Optimality condition for original problem (\ref{eq: main-L1}) is shown in the bottom left.}
\label{fig: opt}
\end{figure}

\section{Review of Fourier Series, Continuous and Discrete}
Here we review some concepts from Fourier analysis and make some connections between the discrete and continuous settings.
\begin{definition}
Let $f\in L_2[0,2\pi]$, so that the Fourier series of $f$ converges (in the $L_2$ sense), and is defined by
\begin{equation}\label{fourier-series}
f(x) = \sum_{k\in \Z} \hat f_k e^{i k x } ,
\end{equation}
where $\hat f_k$ are called the Fourier coefficients and are given by
\begin{equation}\label{fourier-coef}
\hat f_k\big|_{[0,2\pi]} =  \langle f , e^{i k x }  \rangle = \frac{1}{2\pi} \int_0^{2\pi} f(x) e^{-ik x } \, dx.
\end{equation}
\end{definition}

Note that this definition could be equivalently given for any interval of length $2\pi$, e.g. $[-\pi,\pi)$. The classical Fourier series as defined by (\ref{fourier-series}) and (\ref{fourier-coef}) is extremely well-known and studied.  Likewise, many are also familiar with the related discrete Fourier transform (DFT), which is perhaps more widely recognized by the fast algorithm on which it is typically computed, the fast Fourier transform (FFT).  Unfortunately the bridge between the two is sometimes not clearly connected in certain disciplines, although there exists a very clear relationship.  Of course, one could infer that the DFT is simply the discretization of (\ref{fourier-series}) and (\ref{fourier-coef}) whenever $f$ is only defined over a discrete set of points.  Essentially this is correct, however we will try to explain several ways in which this could be interpreted.

To understand this a bit more, an alternative definition for the Fourier transform for a function over $[0,R]$ is first provided.
\begin{definition}
Let $f\in L_2[0, R]$, so that the Fourier series of $f$ converges (in the $L_2$ sense), and is defined by
\begin{equation}\label{fourier-series-alt}
f(x) = \sum_{k\in \Z} \hat f_k e^{i\frac{2\pi}{R} k x } ,
\end{equation}
where $\hat f_k$ are called the Fourier coefficients and are given by
\begin{equation}\label{fourier-coef-alt}
\hat f_k \big|_{[0,R]} =  \langle f , e^{i2\pi k x }  \rangle = R^{-1}  \int_0^R f(x) e^{-i\frac{2\pi}{R} k x } \, dx.
\end{equation}
\end{definition}

Defining a Fourier series on $[0,2\pi]$, $[0,R]$, or more generally $[a,b]$, essentially are all the same, by simply redefining the coordinate system.  Indeed, observe that (\ref{fourier-coef-alt}) can be written precisely as (\ref{fourier-coef}) by letting $\frac{N}{2\pi} y = x$:
\begin{equation}
\hat f_k \big|_{[0,R]} = \frac{1}{2\pi}  \int_0^{2\pi}  f\left( \tfrac{R}{2\pi} y \right) e^{-i k y } \, dy = \hat g_k \big|_{[0,2\pi]},
\end{equation}
where $g(x) =   f\left( \tfrac{R}{2\pi} x\right)$.
We point this out because we will sometimes change the convention we are using to simplify other aspects of a particular problem.  In this case we are using asymmetric intervals for convenience in matching the DFT, but in the main article we primarily use symmetric intervals, e.g. $[-R,R]$.

\begin{definition}
Let $f \in \C^N$.  Then the DFT of $f$, for $k=0,1,\dots ,N-1$, is defined by
\begin{equation}\label{DFT}
\hat f_k = \sum_{j=0}^{N-1} f_j e^{-i2\pi k \frac{j}{N}}.
\end{equation}
\end{definition}
This definition yields the orthogonal DFT where FFT algorithms can be applied, but more generally can be defined for any $k\in \R$ and can be computed by more recent developments of nonuniform FFT's (NUFFT).

With the way it is written in (\ref{DFT}), if we suppose the $N$-point function $f$ is on the interval $[0,2\pi]$, and therefore $f_j$ approximates $f$ at $\frac{2\pi j}{N}$, then we see that 
\begin{equation}\label{DFT-1}
\hat f_k \big|_{[0,2\pi ]}  = \frac{1}{2\pi} \int_0^{2\pi}  f(x) e^{-i k x } \, dx \approx N^{-1} \sum_{j=0}^{N-1} f\left( \frac{2\pi j}{N}\right ) e^{-i k \left( \frac{2\pi j}{N} \right) } ,
\end{equation}
In other words, if we consider our discrete signals to be defined over $[0,2\pi]$, then coefficients in the DFT can be considered a Riemann sum of the classical Fourier coefficients.

Alternatively, we may suppose our discrete signal is defined over $[0,R]$, i.e. the distance between the discrete points on the $N$ point mesh in $R/N$.  Then we then make the alternative interpretation of the DFT approximation of the continuous transform: 
\begin{equation}\label{DFT-2}
\begin{split}
\hat f_k \big|_{[0,R]} 
& = \int_0^R f(x) e^{-i\frac{2\pi}{R} k x } \, \rmd x \\
& \approx \sum_{j=0}^{N-1} f\left( j\tfrac{R}{N}\right) e^{-i\frac{2\pi}{R} k\left( j\frac{R}{N} \right)}
=  \sum_{j=0}^{N-1} f_j e^{-i\frac{2\pi}{N} k j }.
\end{split}
\end{equation}
The last line is a DFT and the previous approximation is its interpretation as a Riemann sum of the original integral.  In summary, the discrete approximations in (\ref{DFT-1}) and (\ref{DFT-2}) are the same as (\ref{DFT}), where there's simply a different interpretation of the location of $f_j$ on the Cartesian grid and the corresponding frequency.  Many of these ideas may be naturally extended to the 2D coordinate system, which the reader may find as a useful exercise for the contents of the main article.


\begin{thebibliography}{10}

\bibitem{achim2003sar}
A.~Achim, P.~Tsakalides, and A.~Bezerianos.
\newblock {SAR} image denoising via {B}ayesian wavelet shrinkage based on
  heavy-tailed modeling.
\newblock {\em IEEE Transactions on Geoscience and Remote Sensing},
  41(8):1773--1784, 2003.

\bibitem{Andersson}
F.~Andersson, R.~Moses, and F.~Natterer.
\newblock Fast {F}ourier methods for synthetic aperture radar imaging.
\newblock {\em IEEE Transactions on Aerospace and Electronic Systems},
  48(1):215--229, 2012.

\bibitem{Averbuch2006145}
A.~Averbuch, R.~Coifman, D.~Donoho, M.~Elad, and M.~Israeli.
\newblock Fast and accurate polar {F}ourier transform.
\newblock {\em Applied and Computational Harmonic Analysis}, 21(2):145 -- 167,
  2006.

\bibitem{barzilai1988two}
J.~Barzilai and J.~M. Borwein.
\newblock Two-point step size gradient methods.
\newblock {\em IMA journal of numerical analysis}, 8(1):141--148, 1988.

\bibitem{bredies2010total}
K.~Bredies, K.~Kunisch, and T.~Pock.
\newblock Total generalized variation.
\newblock {\em SIAM Journal on Imaging Sciences}, 3(3):492--526, 2010.

\bibitem{bruckner1997real}
A.~M. Bruckner, J.~B. Bruckner, and B.~S. Thomson.
\newblock {\em Real analysis}.
\newblock ClassicalRealAnalysis. com, 1997.

\bibitem{Gotcha}
C.~H. Casteel, L.~A. Gorham, M.~J. Minardi, S.~M. Scarborough, K.~D. Naidu, and
  U.~K. Majumder.
\newblock A challenge problem for 2{D}/3{D} imaging of targets from a
  volumetric data set in an urban environment.
\newblock 6568:65680D, 2007.

\bibitem{cheney2001mathematical}
M.~Cheney.
\newblock A mathematical tutorial on synthetic aperture radar.
\newblock {\em SIAM review}, 43(2):301--312, 2001.

\bibitem{cheney2009fundamentals}
M.~Cheney and B.~Borden.
\newblock {\em Fundamentals of radar imaging}, volume~79.
\newblock Siam, 2009.

\bibitem{cody1992fast}
M.~A. Cody.
\newblock The fast wavelet transform: Beyond {F}ourier transforms.
\newblock {\em Dr. Dobb's Journal}, 17(4):16--28, 1992.

\bibitem{cutrona1990synthetic}
L.~Cutrona.
\newblock Synthetic aperture radar.
\newblock {\em Radar handbook}, 2:2333--2346, 1990.

\bibitem{dainty2013laser}
J.~C. Dainty.
\newblock {\em Laser speckle and related phenomena}, volume~9.
\newblock Springer Science \& Business Media, 2013.

\bibitem{dungan2010civilian}
K.~E. Dungan, C.~Austin, J.~Nehrbass, and L.~C. Potter.
\newblock Civilian vehicle radar data domes.
\newblock {\em Algorithms for synthetic aperture radar Imagery XVII}, 7699(1),
  2010.

\bibitem{1166689}
J.~Fessler and B.~Sutton.
\newblock Nonuniform fast {F}ourier transforms using min-max interpolation.
\newblock {\em IEEE Transactions on Signal Processing}, 51(2):560--574, 2003.

\bibitem{gilman2015mathematical}
M.~Gilman and S.~Tsynkov.
\newblock A mathematical model for {SAR} imaging beyond the first {B}orn
  approximation.
\newblock {\em SIAM Journal on Imaging Sciences}, 8(1):186--225, 2015.

\bibitem{bregman}
T.~Goldstein and S.~Osher.
\newblock The split {B}regman method for {L}1-regularized problems.
\newblock {\em SIAM Journal on Imaging Sciences}, 2(2):323--343, 2009.

\bibitem{goodman2007speckle}
J.~W. Goodman.
\newblock {\em Speckle phenomena in optics: theory and applications}.
\newblock Roberts and Company Publishers, 2007.

\bibitem{SARMATLAB}
L.~A. Gorham and L.~J. Moore.
\newblock {SAR} image formation toolbox for {MATLAB}.
\newblock In {\em SPIE Defense, Security, and Sensing}, pages 769906--769906,
  2010.

\bibitem{Greengard}
L.~Greengard and J.~Y. Lee.
\newblock Accelerating the nonuniform fast {F}ourier transform.
\newblock {\em SIAM Review}, 46(3):443--454, 2004.

\bibitem{hestenes1969multiplier}
M.~R. Hestenes.
\newblock Multiplier and gradient methods.
\newblock {\em Journal of optimization theory and applications}, 4(5):303--320,
  1969.

\bibitem{Jakowatz1996}
C.~V. Jakowatz, D.~E. Wahl, P.~H. Eichel, D.~C. Ghiglia, and P.~A. Thompson.
\newblock {\em A Tomographic Foundation for Spotlight-Mode SAR Imaging}, pages
  33--103.
\newblock Springer US, Boston, MA, 1996.

\bibitem{kaipio2006statistical}
J.~Kaipio and E.~Somersalo.
\newblock {\em Statistical and computational inverse problems}, volume 160.
\newblock Springer Science \& Business Media, 2006.

\bibitem{li2010efficient}
C.~Li.
\newblock {\em An efficient algorithm for total variation regularization with
  applications to the single pixel camera and compressive sensing}.
\newblock PhD thesis, Rice University, 2010.

\bibitem{Li2013}
C.~Li, W.~Yin, H.~Jiang, and Y.~Zhang.
\newblock An efficient augmented {L}agrangian method with applications to total
  variation minimization.
\newblock {\em Comput. Optim. Appl.}, 56(3):507--530, 2013.

\bibitem{munson1984importance}
D.~Munson and J.~Sanz.
\newblock The importance of random phase for image reconstruction from
  frequency offset fourier data.
\newblock In {\em Acoustics, Speech, and Signal Processing, IEEE International
  Conference on ICASSP'84.}, volume~9, pages 158--161. IEEE, 1984.

\bibitem{munson1983tomographic}
D.~C. Munson, J.~D. O'Brien, and W.~K. Jenkins.
\newblock A tomographic formulation of spotlight-mode synthetic aperture radar.
\newblock {\em Proceedings of the IEEE}, 71(8):917--925, 1983.

\bibitem{munson1984image}
D.~C. Munson and J.~L. Sanz.
\newblock Image reconstruction from frequency-offset {F}ourier data.
\newblock {\em Proceedings of the IEEE}, 72(6):661--669, 1984.

\bibitem{munson86}
D.~C. Munson, Jr. and J.~L.~C. Sanz.
\newblock Phase-only image reconstruction from offset {F}ourier data.
\newblock {\em Optical Engineering}, 25(5):255655--255655--, 1986.

\bibitem{Nat-Cheney}
F.~Natterer, M.~Cheney, and B.~Borden.
\newblock Resolution for radar and {X}-ray tomography.
\newblock {\em Inverse Problems}, 19(6):S55, 2003.

\bibitem{nolan2002synthetic}
C.~J. Nolan and M.~Cheney.
\newblock Synthetic aperture inversion.
\newblock {\em Inverse Problems}, 18(1):221, 2002.

\bibitem{pinel2013electromagnetic}
N.~Pinel and C.~Boulier.
\newblock {\em Electromagnetic wave scattering from random rough surfaces:
  Asymptotic models}.
\newblock John Wiley \& Sons, 2013.

\bibitem{rigling2005taylor}
B.~D. Rigling and R.~L. Moses.
\newblock Taylor expansion of the differential range for monostatic {SAR}.
\newblock {\em IEEE Transactions on Aerospace and Electronic Systems},
  41(1):60--64, 2005.

\bibitem{toby-web}
T.~Sanders.
\newblock Matlab imaging algorithms: Image reconstruction, restoration, and
  alignment, with a focus in tomography.
\newblock \url{http://www.toby-sanders.com/software},
  \url{https://doi.org/10.13140/RG.2.2.33492.60801}.
\newblock Accessed: 2016-19-08.

\bibitem{sanders2017composite}
T.~Sanders, A.~Gelb, and R.~B. Platte.
\newblock Composite {SAR} imaging using sequential joint sparsity.
\newblock {\em Journal of Computational Physics}, 338:357--370, 2017.

\bibitem{sanders2018multiscale}
T.~Sanders and R.~B. Platte.
\newblock Multiscale higher-order {TV} operators for {L}1 regularization.
\newblock {\em Advanced structural and chemical imaging}, 4(1):12, 2018.

\bibitem{steidl2004equivalence}
G.~Steidl, J.~Weickert, T.~Brox, P.~Mr{\'a}zek, and M.~Welk.
\newblock On the equivalence of soft wavelet shrinkage, total variation
  diffusion, total variation regularization, and sides.
\newblock {\em SIAM Journal on Numerical Analysis}, 42(2):686--713, 2004.

\bibitem{thompson1996spotlight}
P.~Thompson, D.~E. Wahl, P.~H. Eichel, D.~C. Ghiglia, and C.~V. Jakowatz.
\newblock Spotlight-mode synthetic aperture radar: A signal processing
  approach.
\newblock 1996.

\bibitem{tibshirani2013lasso}
R.~J. Tibshirani et~al.
\newblock The lasso problem and uniqueness.
\newblock {\em Electronic Journal of Statistics}, 7:1456--1490, 2013.

\bibitem{wu2010augmented}
C.~Wu and X.~C. Tai.
\newblock Augmented {L}agrangian method, dual methods, and split {B}regman
  iteration for {ROF}, vectorial {TV}, and high order models.
\newblock {\em SIAM Journal on Imaging Sciences}, 3(3):300--339, 2010.

\end{thebibliography}
\end{document}